\documentclass[reqno, 11pt]{amsart}
\usepackage{amssymb, empheq}
\usepackage{graphicx}
\usepackage[usenames, dvipsnames]{color}

\usepackage{verbatim}

\newcommand{\mysection}[1]{\section{#1}
\setcounter{equation}{0}}

\newtheorem{theorem}{Theorem}[section]

\newtheorem{lemma}[theorem]{Lemma}
\newtheorem{proposition}[theorem]{Proposition}

\theoremstyle{definition}
\newtheorem{remark}[theorem]{Remark}

\theoremstyle{definition}

\theoremstyle{definition}

\makeatletter
\def\dashint{\operatorname%
{\,\,\text{\bf--}\kern-.98em\DOTSI\intop\ilimits@\!\!}}
\makeatother

\def\l{\mathcal{L}}

\def\bR{\mathbb{R}}

\def\cD{\mathcal{D}}

\begin{document}
\title[Schauder Estimates]{The Schauder estimates for higher-order parabolic systems with time irregular coefficients}

\author[H. Dong]{Hongjie Dong}
\address[H. Dong]{Division of Applied Mathematics, Brown University,
182 George Street, Providence, RI 02912, USA}
\email{Hongjie\_Dong@brown.edu}
\thanks{H. Dong was partially supported by the NSF under agreement DMS-1056737.}

\author[H. Zhang]{Hong Zhang}
\address[H. Zhang]{Division of Applied Mathematics, Brown University,
182 George Street, Providence, RI 02912, USA}
\email{Hong\_Zhang@brown.edu}
\thanks{H. Zhang was partially supported by the NSF under agreement DMS-1056737.}



\begin{abstract}
We prove Schauder estimates for solutions to both divergence and non-divergence type higher-order parabolic systems in the whole space and the half space. We also provide an existence result for divergence type systems in a cylindrical domain. All coefficients are assumed to be only measurable in the time variable and H\"{o}lder continuous in the spatial variables.
\end{abstract}

\maketitle



\mysection{Introduction}
This paper is devoted to the Schauder estimates for divergence and non-divergence type higher-order parabolic systems.

It is well known that the Schauder estimates play an important role in the existence and regularity theories for elliptic and parabolic equations and systems, which have been studied by many authors.  For parabolic equations with constant coefficients, the classical approach is to study the fundamental solutions, see e.g., \cite{Kry96, FA}. Moreover,  combining the classical approach with the argument of freezing the coefficients, we can deal with general parabolic equations with smooth coefficients, for instance see \cite{Kry96, SL97, ADN64}. For systems it has become customary to use Campanato's technique, first introduced in \cite{Cam66}.  Giaquinta \cite{Gia93} provided a comprehensive explanation for the application of this technique to second-order elliptic systems.  Schlag \cite{Sh96} applied this technique to second-order parabolic systems. For higher-order systems, we also refer the reader to \cite[Chap. 3]{LA95} and the references therein.

The classical Schauder estimates were established under the assumption that coefficients are regular in both space and time. In this paper, we consider the coefficients which are regular only with respect to spatial variables. This type of coefficient has been studied by several authors mostly for second-order equations; see, for instance, \cite{BA69, BK80, L92, L00, KP10}. In \cite{L92, L96} Lieberman studied interior and boundary Schauder estimates for second-order parabolic equations with time irregular coefficients using the maximal principle and a Campanato type approach.  More recently,  Boccia \cite{Sb12} considered higher-order non-divergence type parabolic systems in the whole space.

To present our results precisely, let
\begin{equation*}
Lu=\sum_{|\alpha|\le m,\, |\beta|\le m}A^{\alpha\beta}D^\alpha D^\beta u,\quad \l u=\sum_{|\alpha|\le m,\,|\beta|\le m}D^\alpha(A^{\alpha\beta}D^\beta u),
\end{equation*}
where $m$ is a positive integer,
\begin{equation*}
D^\alpha=D_1^{\alpha_1}\cdots D_d^{\alpha_d},\quad \alpha=(\alpha_1,\cdots,\alpha_d),
\end{equation*}
and, for each $\alpha$ and $\beta$, $A^{\alpha\beta}=[A^{\alpha\beta}_{ij}]_{i,j=1}^n$ is an $n\times n$ real matrix-valued function. Moreover the leading coefficients satisfy the so-called Legendre--Hadamard ellipticity condition (see \eqref{hadamard}), which is more general than the strong ellipticity condition.  The functions used throughout this paper
\begin{equation*}
u=(u^1,\cdots,u^n)^{tr},\quad f=(f^1,\cdots,f^n)^{tr},\quad f_\alpha=(f_\alpha^1,\cdots,f_\alpha^n)^{tr}
\end{equation*}
are real vector-valued functions. The parabolic systems which we study are
\begin{equation*}
u_t+(-1)^m Lu=f,\quad u_t+(-1)^m\l u=\sum_{|\alpha|\le m}D^\alpha f_\alpha
\end{equation*}
in the whole space, or in the half space or cylindrical domains with the Dirichlet boundary condition $u=|Du|=\ldots=|D^{m-1}u|=0$,
where the first system is in the non-divergence form and the second system is in the divergence form.

{We assume that all the coefficients and data are H\"older continuous with respect to the spatial variables and merely measurable with respect to the time variable.}
For the non-divergence form systems, we prove that $D^{2m} u$ is H\"{o}lder continuous in the whole space, and in the half space all the $2m$th order derivatives of $u$ are H\"{o}lder continuous up to the boundary with the exception of  $D_d^{2m} u$, where $x_d$ is the normal direction of the boundary. For the divergence form systems, we prove that all the $m$th order derivatives are H\"{o}lder continuous up to the boundary.  We also prove an existence theorem for the divergence form systems in a cylindrical domain, provided that the boundary of the domain is sufficiently smooth. To our best knowledge, these results are new for higher-order systems and they extend the corresponding results found in Lieberman \cite{L92} for second-order scalar equations. In the special case of second-order parabolic systems, compared to \cite{Sh96} our conditions on the coefficients and data are more general. In particular, we do not require the data to be vanish on the lateral boundary of the domain, i.e., the compatibility condition in \cite{Sh96}.

For the proof, we use some results in \cite{DK10}, in which Dong and Kim proved the $L_p$ estimates for the divergence and non-divergence type higher-order parabolic systems, with the same type of coefficients considered in this paper. Let us give the outline of the proofs. In the divergence case, with the classical $L_2$ estimates, we prove
\begin{align*}
&\int_{Q_r(X_0)}|D^m u-(D^m u)_{Q_r(X_0)}|^2 dx\,dt \\
&\le C(\frac{r}{R})^{2+d+2m} \int_{Q_R(X_0)}|D^m u-(D^m u)_{Q_R(X_0)}|^2 dx\,dt,\quad\forall \, r\le R,
\end{align*}
where $u$ is a solution of
\begin{equation*}
u_t+(-1)^m \l_0 u=0\quad \text{in}\quad Q_{2R}.
\end{equation*}
Here
\begin{equation*}
\l_0=\sum_{|\alpha|=|\beta|=m}D^\alpha(A^{\alpha\beta}(t)D^\beta),
\end{equation*}
and $X_0\in Q_R$.
The coefficients $A^{\alpha\beta}$ which here depend only on $t$, are called simple coefficients. For  {a} similar inequality corresponding to the boundary estimates, we combine the $L_p$ estimates established in \cite{DK10} and the Sobolev embedding theorem to prove the H\"older continuity and obtain, for instance, the following mean oscillation type estimate {for systems} with simple coefficients,
\begin{align*}
&\int_{Q^+_r(X_0)}|D_d^mu-(D_d^mu)_{Q^+_r(X_0)}|^2\,dx\,dt\\
&\le C(\frac{r}{R})^{2\gamma+d+2m}\int_{Q^+_R(X_0)}|D_d^mu-(D_d^mu)_{Q_R^+(X_0)}|^2\,dx\,dt,\quad\forall \,{r\le R},
\end{align*}
where $X_0\in \{x_d=0\}\cap Q_{R}$ {and} $\gamma\in(0,1)$ {are arbitrary}, and $u$ satisfies
\begin{equation*}
u_t+(-1)^m \l_0 u=0\quad \text{in}\quad Q^+_{2R}
\end{equation*}
with the Dirichlet boundary condition $u=|Du|=\ldots=|D^{m-1}u|=0$ on $\{x_d=0\}$.

Similar interior and boundary estimates for the non-divergence form systems can be established in the same fashion.
In particular, we can always differentiate the system
$$
\left\{
\begin{aligned}
                \label{eq11}
&u_t+(-1)^m L_0 u=0 \quad \text{in} \quad Q_{2R}^+\\
&u=0,\,D_du=0,\cdots, D_d^{m-1}u=0\quad \text{on}\quad \{x_d=0\}\cap Q_{2R}
\end{aligned}\right.
$$
with respect to tangential direction $x^\prime$. This together with the $W_p^{1,2m}$ estimates implies that $D_{x'}D^{2m-1}u$ is in some H\"older space $C^{\frac{a}{2m},a}(Q_R^+)$ with $a$ arbitrarily close to 1 from below, which {yields} the mean oscillation type estimates.

For the coefficients which depend on both $t$ and $x$, we use the standard argument of freezing the coefficients to obtain the Campanato type estimates and then achieve the estimates of the H\"{o}lder norms.

The paper is organized as follows. In {the} next section we introduce some notation and state our main results. Section 3 is devoted to some necessary technical lemmas. In Section 4 we make necessary preparations and in Section 5 prove our main result {for divergence type systems}, Theorem \ref{existence}. Section 6 deals with non-divergence type systems. 
\mysection{main results}
We first introduce some notation used throughout the paper. A point in $\mathbb{R}^d$ is denoted by $x=(x_1,\cdots,x_d)$. We also denote $x=(x^\prime,x_d)$, where $x^\prime \in \mathbb{R}^{{d-1}}$. A point in
\begin{equation*}
\mathbb{R}^{d+1}=\mathbb{R}\times\mathbb{R}^d=\{(t,x):t\in \mathbb{R},x\in \mathbb{R}^d\}
\end{equation*}
is denoted by $X=(t,x)$.~For $T \in (-\infty,\infty]$, set
\begin{equation*}
\mathcal{O}_T=(-\infty,T)\times \mathbb{R}^d, \quad \mathcal{O}^+_T=(-\infty,T)\times\mathbb{R}^d_+,
\end{equation*}
where $\mathbb{R}^d_+=\{x=(x_1,\cdots,x_d)\in \mathbb{R}^d: x_d>0\}$. In particular, when $T=\infty$, we use $\mathcal{O}_\infty^+=\mathbb{R}\times \mathbb{R}^d_+$. {Denote}
\begin{align*}
&B_r(x_0)=\{x\in \mathbb{R}^d: |x-x_0|<r\},\quad Q_r(t_0,x_0)=(t_0-r^{2m},t_0)\times B_r(x_0),\\
&B_r^+(x_0)=B_r(x_0)\cap \bR^d_+,\qquad Q^+_r(t_0,x_0)=Q_r(t_0,x_0) \cap\mathcal{O}_\infty^+.
\end{align*}
{We use the abbreviations, for example, $B_r$ if $x_0 = 0$ and $Q_r$ if $(t_0,x_0) = (0,0)$}. The parabolic boundary of $Q_r{(t_0,x_0)}$ is defined to be
\begin{equation*}
\partial_p Q_r(t_0,x_0)=[t_0,t_0-r^{2m}) \times \partial B_r(x_0)\cup \{t=t_0-r^{2m}\} \times B_r(x_0).
\end{equation*}
The parabolic boundary of $Q^+_r(t_0,x_0)$ is
\begin{equation*}
\partial_p Q^+_r(t_0,x_0)=(\partial_p Q_r(t_0,x_0)\cap \overline{\mathcal{O}_\infty^+} )\cup(Q_r(t_0,x_0)\cap \{x_d =0\}).
\end{equation*}
We denote
\begin{equation*}
\langle f,g\rangle_\Omega=\int_\Omega f^{tr}g=\sum_{j=1}^n\int_\Omega f^jg^j.
\end{equation*}
For a function $f$ defined on $\mathcal{D} \subset \mathbb{R}^{d+1}$, we set

\begin{align*}
&[f]_{{a,b},\mathcal{D}}\equiv \sup\left\{\frac{|f(t,x)-f(s,y)|}{|t-s|^{{a}}+|x-y|^{b}} :(t,x),(s,y)\in \mathcal{D},\, (t,x)\neq (s,y) \right\},
\end{align*}
where $a,b\in (0,1].$ The H\"{o}lder semi-norm with respect to $t$ is denoted by
\begin{equation*}
\langle f\rangle_{a,\mathcal{D}}\equiv \sup \left\{\frac{|f(t,x)-f(s,x)|}{|t-s|^{a}}:(t,x),(s,x)\in \mathcal{D},\, t\neq s\right\},
\end{equation*}
where $a\in (0,1]$. We will also use the H\"{o}lder semi-norm with respect to $x$
\begin{equation*}
[f]_{a,\mathcal{D}}^\ast \equiv \sup\left\{\frac{|f(t,x)-f(t,y)|}{|x-y|^a} :(t,x),(t,y)\in \mathcal{D}, x\neq y \right\},
\end{equation*}
and denote
\begin{equation*}
\|f\|^\ast_{a,\mathcal{D}}=\|f\|_{L_\infty{(\cD)}}+[f]^\ast_{a,\mathcal{D}},
\end{equation*}
where $a\in (0,1]$.
The space corresponding to $\|\cdot\|_{a,{\cD}}^*$ is denoted by  $C^{a\ast }(\mathcal{D})$.
For $a\in (0, 1]$, set
\begin{equation*}
\|f\|_{\frac{a}{2m},a,\mathcal{D}}=\|f\|_{L_\infty(\cD)}+[f]_{\frac{a}{2m},a,\mathcal{D}}.
\end{equation*}
The space corresponding to $\|\cdot\|_{{\frac{a}{2m},} a,\mathcal{D}}$ is denoted by $C^{\frac{a}{2m},a}(\mathcal{D})$. For $a\in (1,2m)$, not an integer, we define
\begin{equation*}
\|f\|_{\frac{a}{2m},a,\mathcal{D}}=\|f\|_{L_\infty{(\cD)}}+\sum_{|\alpha|< a}[D^\alpha f]_{\frac{a-|\alpha|}{2m},\{a\},\cD},
\end{equation*}
where $\{a\}=a-[a]$.
We use the following Sobolev space
\begin{equation*}
W_p^{1,2m}((S,T)\times \Omega)=\{u: u_t, D^\alpha u \in L_p((S,T)\times \Omega),  0\le|\alpha|\le 2m\}.
\end{equation*}

We denote the average of $f$ in $\cD$ to be
$$
(f)_\mathcal{D}=\frac{1}{|\mathcal{D}|}\int_{\mathcal{D}} f(t,x)\,dx\,dt=\dashint_\mathcal{D}f(t,x)\,dx\,dt.\\
$$
Sometimes we take average only with respect to $x$. For instance,
 \begin{equation*}
  (f)_{B_R(x_0)}(t)=\dashint_{B_R(x_0)}f(t,x)\,dx.
 \end{equation*}

Throughout this paper, we assume that all the coefficients are measurable and bounded,
\begin{equation*}
|A^{\alpha\beta}|\le K.
\end{equation*}
In addition, we impose the Legendre--Hadamard ellipticity condition with constant $\lambda>0$ on the leading coefficients, i.e.,
\begin{equation}\label{hadamard}
 \sum_{|\alpha|=|\beta|=m}A^{\alpha\beta}_{ij} \xi_i\xi_j\ \eta^\alpha\eta^\beta\ge \lambda |\xi|^2|\eta|^{2m},
\end{equation}
where $\xi \in \mathbb{R}^n$, $\eta\in \mathbb{R}^d$, and $\eta^\alpha=\eta_1^{\alpha_1}\eta_2^{\alpha_2}\cdots\eta_d^{\alpha_d}$.
Here we call $A^{\alpha\beta}_{ij}$ the leading coefficients if $|\alpha|=|\beta|=m $.

Throughout this paper
\begin{equation*}
\l_0u=L_0u=\sum_{|\alpha|=|\beta|=m}D^\alpha (A^{\alpha\beta}(t)D^\beta u),
\end{equation*}
where $A^{\alpha\beta}$ are measurable in $t$ and satisfy \eqref{hadamard}.

We are ready to state the main results of the paper. The first result is about the Schauder estimate and solvability for divergence type higher-order systems in cylindrical domains.

\begin{theorem}\label{existence}
Assume $ f_\alpha \in C^{a\ast}$ for $|\alpha|=m$ and $f_\alpha \in L_\infty$ for $|\alpha|<m$, where $a \in (0,1)$. Suppose that the operator $\l $ satisfies the Legendre--Hadamard condition, i.e., \eqref{hadamard}, and $A^{\alpha \beta} \in C^{a\ast}$. Let $g$ be a smooth function in $\mathbb{R}^{d+1}$ and $\Omega \in C^{m,a}$. Then
$$
\left\{\begin{aligned}
&u_t+(-1)^m \l u =\sum_{|\alpha|\le m}D^\alpha f_\alpha \quad \text{in} \quad (0,T)\times \Omega ,\\
&u=g, \,D u=D g,\,\cdots,\,D^{m-1} u=D^{m-1} g \quad \text{on}\quad  [0,T)\times \partial\Omega,\\
&u=g\quad \text{on} \quad\{0\}\times \Omega
\end{aligned}
\right.
$$
has a unique solution $u\in C^{\frac{a+m}{2m},a+m}([0,T)\times \bar \Omega)$. Moreover there exists a constant $C=C(d,n,m,\lambda,K, \|A^{\alpha\beta}\|_a^\ast,\Omega, a)$ such that,
\begin{equation}\label{finaldiv}
\|u\|_{\frac{a+m}{2m},a+m,(0,T)\times \Omega} \le C{(\|u\|_{L_2((0,T)\times \Omega)}+F+G)},
\end{equation}
where $$F=\sum_{|\alpha|=m}[f_\alpha]_a^\ast+\sum_{|\alpha|<m}\|f_\alpha\|_{L_\infty}$$ and $$G=\sum_{|\alpha|=m}\|D^\alpha g\|_a^\ast+\sum_{|\alpha|<m}\|D^\alpha g\|_{L_\infty}+\|g_t\|_{L_\infty}.$$
\end{theorem}

The following theorem is regarding a priori interior and boundary Schauder estimates for the non-divergence type systems.
\begin{theorem}\label{nondiv theorem}
Suppose that $A^{\alpha\beta}\in C^{a\ast}$ and $f \in C^{a\ast}$, where $a\in(0,1)$. Let $u \in C^\infty_{loc}(\bR^{d+1})$ be a solution to
\begin{equation}
u_t+(-1)^m Lu=f,\label{nondivequation}
\end{equation}
where $L$ satisfies the Legendre--Hadamard condition. For any $R \le 1$,

\noindent I. Interior case: if \eqref{nondivequation} holds in $Q_{4R}$, then there exists a constant $C=C(d,n,m,\lambda,K,\|A^{\alpha\beta}\|_a^\ast,R, a)$ such that
\begin{equation*}
\|u_t\|_{a, Q_R}^*+{\|D^{2m}u\|_{\frac{a}{2m},a,Q_R}}\le C({\|f\|_{a, Q_{4R}}^{\ast}}+\|u\|_{L_2(Q_{4R})});
\end{equation*}

\noindent II. Boundary case: if \eqref{nondivequation} holds in $Q^+_{4R}$ with the Dirichlet boundary condition on $\{x_d=0\}\cap Q_{4R}$, then there exists a constant $C=C(d,n,m,\lambda,K, \|A^{\alpha\beta}\|_a^\ast, R, a)$  such that
\begin{equation*}
[D_{x^\prime}D^{2m-1}u]_{\frac{a}{2m},a,Q^+_R}
\le C[f]_{a, Q_{4R}^+}^{\ast}+C\sum_{|\gamma|\le 2m}\|D^{\gamma}u\|_{L_\infty(Q^+_{4R})}.
\end{equation*}
\end{theorem}

We note that the boundary estimate in Theorem \ref{nondiv theorem} is optimal even for the heat equation, in the sense that near the boundary $D_d^{2}u$ might be discontinuous with respect to $x$  if $f$ has jump discontinuities in $t$. See, for instance, \cite[\S 3]{SV88} and \cite[\S 16]{L92}. In order to estimate the H\"older semi-norm of $D_d^{2m} u$, we need to impose more regularity conditions on $A^{\alpha\beta}$ and $f$. See Remark \ref{rem6.3} below for a discussion.

\mysection{Technical preparation}
We need the following version of Campanato's theorem. The proof can be found in \cite{Gia93} and \cite{Sh96}.
\begin{lemma}\label{holder}
(i) Let $f\in L_2(Q_2)$, $a\in(0,1)$, and assume that
\begin{equation}
\dashint_{Q_r(t_0,x_0)}|f-(f)_{Q_r(t_0,x_0)}|^2\,dx\,dt\le A^2r^{2a},\quad \forall\, (t_0,x_0)\in Q_1,\label{holder1}
\end{equation}
and $0<r \le 1$. Then $f \in C^{\frac{a}{2m},a}(Q_1)$ and
\begin{equation*}
[f]_{\frac{a}{2m},a,Q_1}\le CA,
\end{equation*}
with $C=C(d)$.

ii) Let $f\in L_2(Q_2^+)$, $a\in(0,1)$, and assume that \eqref{holder1} holds for $r<{x_0}_d$. Moreover
\begin{equation*}
\dashint_{Q_r^+(t_1,x_1)}|f-(f)_{Q_r^+(t_1,x_1)}|^2\,dx\,dt\le A^2r^{2a},\quad \forall\,(t_1,x_1) \in \{x_d=0\}\cap Q_2,
\end{equation*}
 and $0<r\le 1$. Then $f \in C^{\frac{a}{2m},a}(Q_1^+)$ and
\begin{equation*}
[f]_{\frac{a}{2m},a,Q^+_1}\le CA,
\end{equation*}
with $C=C(d)$.
\end{lemma}
The following technical lemma will be frequently used in this paper. The proof can be found in \cite{L92} and \cite{Sh96}.
\begin{lemma}\label{tech}
Let $\Phi$ be a nonnegative, nondecreasing function on $(0,r_0]$ such that
\begin{equation*}
\Phi(\rho)\le A(\frac{\rho}{r})^a\Phi(r)+Br^b
\end{equation*}
for any $0<\rho<r\le r_0$, where $0<b<a$ are fixed constants. Then
\begin{equation*}
\Phi(r)\le Cr^b(r_0^{-b}\Phi(r_0)+B)\quad {\forall} r \in (0,r_0)
\end{equation*}
with a constant $C=C(A,a,b).$
\end{lemma}
Here is another technical lemma.
\begin{lemma}\label{tech function}
Assume that $\Omega$ is an open bounded domain in $\bR^d$. Then there exists a function $\xi\in C^{\infty}_0(\Omega)$ with unit integral such that for any $0<|\alpha|\le m$
\begin{equation*}
\int_{\Omega}\xi(y)y^\alpha dy =0.
\end{equation*}
\end{lemma}
\begin{proof}
Let $M=\sum_{l=1}^{m}(^{l+d-1}_{\,\,d-1})$. {We look for $\xi$ in the form
\begin{equation*}
\xi(y)=\sum_{|\beta|\le {m}} a_\beta \xi_\beta(y),
\end{equation*}
where $\beta$ is an $n$-tuple multi-index and $\xi_\beta\in C_0^\infty(\Omega),|\beta|\le {m},$ are functions to be chosen.
The problem is then reduced to the following linear system of $\{a_\beta\}$:
\begin{equation*}
\sum_{|\beta|\le {m}} a_\beta \int_{\Omega}\xi_\beta(y)y^\alpha\,dy=1_{|\alpha|=0}, \quad \forall\,|\alpha|\le m.
\end{equation*}
It suffices to choose $\xi_\beta$ such that the $(M+1)\times (M+1)$ coefficient matrix $$
A:=\big\{\int_{\Omega}\xi_\beta(y) y^\alpha \,dy\big\}_{|\alpha|,|\beta|\le {m}}$$
is invertible. To this end, we use a perturbation argument. It is easily seen that the matrix
$$
\big\{\int_{\Omega}y^\beta y^\alpha \,dy\big\}_{|\alpha|,|\beta|\le {m}}
$$
is invertible. Since $C_0^\infty(\Omega)$ is dense in $L_2(\Omega)$, for any $\varepsilon>0$ and $|\beta|\le {m}$, there exists $\xi_\beta\in C_0^\infty(\Omega)$ such that $\|\xi_\beta-y^\beta\|_{L_2(\Omega)}\le \varepsilon$. By the Cauchy--Schwarz inequality,
$$
\big|\int_{\Omega}\xi_\beta(y) y^\alpha \,dy-\int_{\Omega}y^\beta(y) y^\alpha \,dy\big|\le C\varepsilon,
$$
where $C$ is a positive constant independent of $\varepsilon$. By the continuity of matrix determinant, upon taking $\varepsilon$ sufficiently small, $A$ is {still} invertible. The lemma is proved.}
\end{proof}

\mysection{Estimates for divergence type systems}
This section is devoted to the interior and boundary a priori estimates for divergence type higher-order systems.
\subsection{Interior estimates}
The main result of this subsection is the following proposition.
\begin{proposition}\label{div_complete}
Assume that $R\le 1$ and $u\in C_{loc}^\infty(\bR^{d+1})$ satisfies
\begin{equation*}
u_t+(-1)^m \l u=\sum_{|\alpha|\le m}D^\alpha f_\alpha \quad \text{in}\quad Q_{2R}.
\end{equation*}
Suppose that $A^{\alpha\beta}\in C^{a\ast}$, $f_\alpha\in C^{a\ast}$ if $|\alpha|=m$ where $a\in(0,1)$, and $f_\alpha \in L_\infty$ if $|\alpha|<m$. Then there exists a constant $C(d,n,m,\lambda,K, \|A^{\alpha\beta}\|_a^\ast, R, a)$ such that
\begin{align}
                                \label{eq4.42}
[D^m u]_{\frac{a}{2m},a,Q_R}&\le C(\sum_{|\beta|\le m}\|D^{\beta}u\|_{L_\infty(Q_{2R})}+F),\\
                                \label{eq4.43}
\langle u\rangle_{\frac{1}{2}+\frac{a}{2m}, Q_R} &\le C(\sum_{|\beta|\le m}\|D^{\beta}u\|_{L_\infty(Q_{2R})}+F),
\end{align}
where $$F=\sum_{|\alpha|=m}[f_\alpha]_{a, Q_{2R}}^\ast+\sum_{|\alpha|<m}\|f_\alpha\|_{L_\infty(Q_{2R})}.$$
\end{proposition}

For the proof of the proposition, we first consider homogeneous systems with simple coefficients.

\begin{lemma}\label{fund}
Assume  that u $\in C_{loc}^\infty(\mathbb{R}^{d+1})$ and satisfies
\begin{equation}
\label{hom} u_t+(-1)^m \l_0 u=0 \quad in \quad Q_{2R},
\end{equation}
where $R \in (0,\infty)$.
Then for any $ X_0 \in Q_R$ and $r<R$, There is a constant$~C=C(d,n,m,\lambda)$ such that
\begin{equation}
\int_{Q_r(X_0)}|u-(u)_{Q_r(X_0)}|^2 \,dx\,dt\le C (\frac{r}{R})^{2+2m+d} \int_{Q_R(X_0)} |u-(u)_{Q_R(X_0)}|^2\,dx\,dt.\label{eq 4.08}
\end{equation}
\end{lemma}
\begin{proof}
By scaling and translation of the coordinates, without loss of generality, we can assume that $R=1$ and $X_0=(0,0)$. First we consider the case $r\le 1/4$. Since
\begin{equation*}
\int_{Q_r}|u-(u)_{Q_r}|^2 \,dx\,dt\le C r^{2+2m+d} \sup_{Q_r} |Du|^2+Cr^{4m+2m+d} \sup_{Q_r}|u_t|^2.
\end{equation*}
By Lemma 4.3 in \cite{DK10},
\begin{equation*}
\sup_{Q_{1/4}}(|Du|+|u_t|)\le C \|u\|_{L_2(Q_1)}.
\end{equation*}
Hence,
\begin{equation*}
\int_{Q_r}|u-(u)_{Q_r}|^2\,dx\,dt \le C r^{2+2m+{d}}  \|u\|^2_{L_2(Q_1)}.
\end{equation*}
Clearly, the inequality above holds true {in the case when $r\in [1/4,1)$ as well}. Since $u-(u)_{Q_1}$ also satisfies \eqref{hom}, substituting $u$ by $u-(u)_{Q_1}$, the lemma is proved.
\end{proof}

\begin{lemma}\label{divsim}
Assume that $r<R \le 1$ and $u\in C_{loc}^\infty(\bR^{d+1})$ satisfies
\begin{equation*}
\label{simple} u_t+(-1)^m \l_0 u=\sum_{|\alpha|\le m}D^{\alpha}f_{\alpha} \quad \text{in} \quad Q_{2R},
\end{equation*}
and  $f_\alpha \in C^{a\ast}$, if $|\alpha| =m$ where $a\in(0,1)$, and $f_\alpha \in L_\infty$ if $|\alpha|< m$.
Then for any  $X_0\in Q_R$, there exists a constant $C=C(n,m,d,\lambda)$ such that
\begin{align*}
&\int_{Q_r(X_0)}|D^mu-(D^mu)_{Q_r(X_0)}|^2\,dx\,dt \\&\le C (\frac{r}{R})^{2+2m+d} \int_{Q_R(X_0)}|D^mu-(D^mu)_{Q_R(X_0)}|^2\,dx\,dt
\\&\,\,+C\sum_{|\alpha|=m}[f_\alpha]_{a,Q_{2R}} ^{\ast2} R^{2a+2m+d} +C\sum_{|\alpha|<m} \|f_\alpha\|_{L_\infty(Q_{2R})}^2R^{2+2m+d}.
\end{align*}
\end{lemma}
\begin{proof}
Let $w$ be the weak solution of the following system
$$
\left\{
\begin{aligned}
&w_t+(-1)^m \l_0 w=\sum_{|\alpha|\le m}D^{\alpha}f_{\alpha} \quad \text{in}\quad Q_{2R},\\
&w=0,\, |Dw|=0,\ldots,\,|D^{m-1}w|=0\quad \text{on} \quad \partial_pQ_{2R}.
\end{aligned}
\right.
$$
Multiply $w$ to both sides of the system and integrate over $Q_{2R}$. Due to the homogeneous Dirichlet boundary condition of $w$, we can replace $D^\alpha f_\alpha$ by $D^\alpha (f_\alpha(t,x)-f_\alpha (t,0))$ when $|\alpha|=m$. By Young's inequality, we get, for any $\varepsilon\in (0,1)$,
\begin{align}\nonumber
&\int_{Q_{2R}}|D^m w|^2\,dx\,dt\\&\le C\sum_{|\alpha|=m}\int_{Q_{2R}}(-1)^m( f_\alpha(t,x)-f_\alpha(t,0))D^\alpha w\,dx\,dt\nonumber\\
&\,+\sum_{|\alpha|<m}\int_{Q_{2R}} (-1)^{|\alpha|}f_\alpha D^\alpha w \,dx\,dt\nonumber\\
&\nonumber\le \varepsilon\|D^m w\|^2_{L_2(Q_{2R})}+  \sum_{|\alpha|=m} C(\varepsilon)\| f_\alpha(t,x)-f_\alpha(t,0)\|_{L_2(Q_{2R})}^2\\
&\label{westimate}\,+\sum_{|\alpha|<m} \|f_\alpha\|_{L_\infty(Q_{2R})} \int_{Q_{2R}}|D^\alpha w|\,dx\,dt.
\end{align}
Since$~ w=0, \,|Dw|=0,\ldots,\,|D^{m-1}w|=0 $~on~$ \partial_p Q_{2R}$,  then for $|\alpha|<m$ and $R \le 1$ by the Poincar\'e inequality
\begin{equation*}
\int_{Q_{2R}}|D^\alpha w|^2\,dx\,dt\le CR^2\int_{Q_{2R}} |D^m w|^2\,dx\,dt,
\end{equation*}
where $C$ is a universal constant.  By the Cauchy--Schwarz inequality and the inequality above, we can estimate the last term in \eqref{westimate} as follows:
\begin{align*}
&\|f_\alpha\|_{L_\infty(Q_{2R})} \int_{Q_{2R}}|D^\alpha w|\,dx\,dt\\
&\le \|f_\alpha\|_{L_\infty(Q_{2R})} |Q_{2R}|^{1/2} (\int_{Q_{2R}} |D^\alpha w|^2\,dx\,dt)^{1/2}\\
&\le C\|f_\alpha\|_{L_\infty(Q_{2R})} R|Q_{2R}|^{1/2} (\int_{Q_{2R}} |D^m w|^2\,dx\,dt)^{1/2}\\
&\le C(\varepsilon)\|f_\alpha\|_{L_\infty(Q_{2R})} ^2 R^{2+d+2m}+\varepsilon \int_{Q_{2R}} |D^m w|^2\,dx\,dt.
\end{align*}
Choosing $\varepsilon$ sufficiently small, we obtain
\begin{align}
                                \label{eq5.00}
&\int_{Q_{2R}}|D^m w|^2\,dx\,dt\nonumber\\
&\,\le C\sum_{|\alpha|=m} \| f_\alpha(t,x)-f_\alpha(t,0)\|_{L_2(Q_{2R})}^2+C\sum_{|\alpha|<m}\|f_\alpha\|_{L_\infty(Q_{2R})} ^2 R^{2+d+2m}\nonumber\\
&\,\le C\sum_{|\alpha|=m} [f_\alpha]_{a,Q_{2R}}^{*2}R^{2a+d+2m}+C\sum_{|\alpha|<m}\|f_\alpha\|_{L_\infty(Q_{2R})} ^2 R^{2+d+2m}.
\end{align}

We now temporarily assume that $A^{\alpha\beta}$ and $f_\alpha$ are smooth functions. By the  classical theory, we know that $w$ is smooth.  Consider $v=u-w$, which is a smooth function as well. Then $v$ satisfies
\begin{equation*}
v_t+(-1)^m \l_0 v=0 \quad \text{in} \quad Q_{2R}.
\end{equation*}
Since $D^m v$ satisfies the same system as $v$, applying Lemma \ref{fund} to $D^m v$, we have
\begin{align}
&\int_{Q_r(X_0)}|D^mv-(D^mv)_{Q_r(X_0)}|^2\,dx\,dt \nonumber\\
&\le C(\frac{r}{R})^{2+d+2m} \int_{Q_R(X_0)} |D^mv-(D^mv)_{Q_R(X_0)}|^2\,dx\,dt.     \label{eq8.54}
\end{align}
By the triangle inequality,
\begin{align*}
&\int_{Q_r(X_0)}|D^mu-(D^mu)_{Q_r(X_0)}|^2\,dx\,dt\\
&= \int_{Q_r(X_0)}|D^mw-(D^mw)_{Q_r(X_0)}+D^m v-(D^m v)_{Q_r(X_0)}|^2\,dx\,dt\\
&\le C\int_{Q_r(X_0)} |D^m w|^2\,dx\,dt +C\int_{Q_r(X_0)}|D^mv-(D^mv)_{Q_r(X_0)}|^2\,dx\,dt.
\end{align*}
By \eqref{eq5.00} and \eqref{eq8.54}, we get
\begin{align*}
& \int_{Q_r(X_0)}|D^mu-(D^mu)_{Q_r(X_0)}|^2\,dx\,dt\\&\le  C\int_{Q_r(X_0)} |D^m w|^2\,dx\,dt+C(\frac{r}{R})^{2+d+2m} \int_{Q_R(X_0)} |D^mv-(D^mv)_{Q_R(X_0)}|^2\,dx\,dt\\
&\le C\int_{Q_{2R}} |D^m w|^2\,dx\,dt+C(\frac{r}{R})^{2+d+2m} \int_{Q_R(X_0)} |D^mu-(D^mu)_{Q_R(X_0)}|^2\,dx\,dt\\
&\le C(\frac{r}{R})^{2+d+2m} \int_{Q_R(X_0)} |D^mu-(D^mu)_{Q_R(X_0)}|^2\,dx\,dt\\
&\,\,+ C\sum_{|\alpha|=m}[f_\alpha]_{a,Q_{2R}}^{\ast2} R^{2a+d+2m}+C\sum_{|\alpha|<m}\|f_\alpha\|^2_{L_\infty(Q_{2R})} R^{2+d+2m}.
\end{align*}
Thus the lemma is proved under the additional smoothness assumption. By a standard argument of mollification, it is easily seen that  we can remove the smoothness assumption. 
\end{proof}
Applying Lemma \ref{tech} and Lemma \ref{divsim}, we get the Campanato type estimate
\begin{align*}
&\int_{Q_r(X_0)}|D^mu-(D^mu)_{Q_r(X_0)}|^2 \,dx\,dt\\&\le C(\frac{r}{R})^{2a+d+2m} \int_{Q_R(X_0)}|D^mu-(D^mu)_{Q_R(X_0)}|^2\,dx\,dt\\
&\,\,+C(\sum_{|\alpha|=m}[f_\alpha]_{a,Q_{2R}}^{\ast2}+\sum_{|\alpha|<m}\|f_\alpha\|^2_{L_\infty(Q_{2R})})r^{2a+d+2m}
\end{align*}
for any $r<R\le 1$ and $X_0 \in Q_R$. Then,
\begin{align*}
&\frac{1}{r^{2a+d+2m}}\int_{Q_r(X_0)}|D^mu-(D^mu)_{Q_r(X_0)}|^2\,dx\,dt \\ &\le
 \frac{C}{R^{2a+d+2m}} \int_{Q_R(X_0)}|D^mu-(D^mu)_{Q_R(X_0)}|^2\,dx\,dt+CF^2\\
&\le\frac{C}{R^{2a+2m+d}} \int_{Q_{2R}}|D^m u|^2\,dx\,dt+CF^2.
\end{align*}
Thanks to Lemma \ref{holder}, we get
\begin{equation}\label{fixeq}
[D^m u]_{\frac{a}{2m},a,Q_R}\le C(R^{-(2a+d+2m)}\int_{Q_{2R}}|D^m u|^2\,dx\,dt+F^2)^{\frac{1}{2}}.
\end{equation}
As for the coefficients which also depend on $x$, the estimates follow from the standard argument of freezing the coefficients. Specifically, we first consider the operator $\l$ that only has the highest-order terms.  We fix a $y\in B_R$, and define
\begin{equation*}
\l_{0y}=\sum_{|\alpha|=|\beta|=m}D^\alpha (A^{\alpha\beta}(t,y)D^\beta).
\end{equation*}
Then
\begin{equation*}
u_t+(-1)^m \l_{0y}u=\sum_{|\alpha|\le m}D^\alpha f_\alpha-(-1)^m(\l-\l_{0y})u=\sum_{|\alpha|\le m}D^{\alpha}\tilde{f}_\alpha,
\end{equation*}
where
\begin{align*}
&\tilde{f_\alpha}=f_\alpha+(-1)^{m+1}\Big(\sum_{|\beta|=m}(A^{\alpha \beta}(t,x)-A^{\alpha\beta}(t,y))D^{\beta}u\Big)\qquad \text{ if}\quad|\alpha|=m, \\
&\tilde{f_\alpha}=f_\alpha \quad \text{if}\quad|\alpha|<m.\\
\end{align*}

Applying the method in proving Lemma \ref{divsim}, as \eqref{eq5.00} we obtain
\begin{multline*}
\|D^m w\|^2_{L_2(Q_{2R})}\le C(\sum_{|\alpha|=|\beta|=m}[A^{\alpha\beta}]_a^{\ast 2}\|D^mu\|_{L_\infty(Q_{2R})}^2 R^{2a+2m+d}
\\+\sum_{|\alpha|<m}\|f_\alpha\|_{L_\infty(Q_{2R})}^2R^{2+2m+d}
+\sum_{|\alpha|=m}[f_\alpha]_{a,Q_{2R} }^{\ast 2}R^{2a+2m+d}).
\end{multline*}
Following the proof of \eqref{fixeq}, we reach
\begin{multline}
[D^m u]_{\frac{a}{2m},a, Q_R}\le C(R^{-(2a+d+2m)}\int_{Q_{2R}}|D^m u|^2\,dx\,dt\\+\sum_{|\alpha|=|\beta|=m}[A^{\alpha\beta}]_a^{\ast 2}\|D^mu\|_{L_\infty(Q_{2R})}^2+F^2)^{\frac{1}{2}}.
    \label{highestimate}
\end{multline}

For systems with lower-order terms, we move all the lower-order terms to the right-hand side to get
\begin{align*}
u_t+(-1)^m\l_h u=&\sum_{|\alpha|=m}D^{\alpha}(f_\alpha+(-1)^{m+1}\sum_{|\beta|<m}A^{\alpha\beta}D^\beta u)\\&\,\,
+\sum_{|\alpha|<m}D^\alpha(f_\alpha+(-1)^{m+1}\sum_{|\beta|\le m}A^{\alpha\beta}D^\beta u),
\end{align*}
where $\l_h u$ denotes the sum of the highest-order terms.
With the inequality we just got, we only need to substitute $\hat{f}_\alpha=f_\alpha+(-1)^{m+1}\sum_\beta A^{\alpha\beta}D^{\beta}u$ into the estimate \eqref{highestimate} and notice that
\begin{equation*}
\sum_{|\beta|<m}[A^{\alpha\beta}D^{\beta}u]_{a,Q_{2R}}^\ast \le C\sum_{|\beta|<m}(\|A^{\alpha\beta}\|_{L_\infty}[D^\beta u]_{a,Q_{2R}}^\ast+[A^{\alpha\beta}]_a^\ast\|D^{\beta}u\|_{L_\infty(Q_{2R})}),
\end{equation*}
when $|\alpha|=m$.
An easy calculation then completes the proof of \eqref{eq4.42}.

It remains to {prove} \eqref{eq4.43}.
We estimate $|u(t,x_0)-u(s,x_0)|$, where $(t,x_0)$, $(s,x_0) \in Q_R$.
Let $\rho=|t-s|^{\frac{1}{2m}}$ and $\eta(x)=\frac{1}{\rho^d}\xi(\frac{x}{\rho})$, where $\xi(x)$ is the function in Lemma \ref{tech function}. We define
\begin{align*}
\overline{u}(t,y)= u(t,y)-T_{m,{x_0}}u(t,{y})+u(t,x_0),
\end{align*}
where $T_{m,{x_0}}u(t,{y})$ is the Taylor expansion of $u(t,{y})$ in ${y}$ at $x_0$ up to $m$th order.
 By the triangle inequality,
\begin{align}
                            \label{eq9.26}
&|u(t,x_0)-u(s,x_0)|\le|u(t,x_0)-\int_{B_\rho(x_0)}
\eta(x_0-y)\overline{u}(t,y)\,dy|\nonumber\\&\,\,+|u(s,x_0)
-\int_{B_\rho(x_0)}\eta(x_0-y)\overline{u}(s,y)\,dy|\nonumber\\
&\,\,+|\int_{B_\rho(x_0)}\eta(x_0-y)\overline{u}(s,y)\,dy-
\int_{B_\rho(x_0)}\eta(x_0-y)\overline{u}(t,y)\,dy|.
\end{align}
The first two terms on the right-hand side of \eqref{eq9.26} can be estimated in a similar fashion: noting that $\int_{B_\rho} \eta(x)\,dx=1$,
\begin{align*}
&|u(t,x_0)-\int_{B_\rho(x_0)}\eta(x_0-y)\overline{u}(t,y)\,dy|\\
&\,=|\int_{B_\rho(x_0)}\eta(x_0-y)(u(t,x_0)-\overline{u}(t,y))\,dy|.
\end{align*}
Since
\begin{align*}
|\overline{u}(t,y)-u(t,x_0)|{=|u(t,y)-T_{m,x_0}u(t,y)|}\le C [D^m u]_{\frac{a}{2m},a, Q_{2R}}|y-x_0|^{m+a},
\end{align*}
we get
\begin{align}
|u(t,x_0)-\int_{B_\rho(x_0)}\eta(x_0-y)\overline{u}(t,y)\,dy|\le C[D^m u]_{\frac{a}{2m},a, Q_{2R}} \rho^{m+a}.\label{eq4.14}
\end{align}
For the last term on the right-hand side of \eqref{eq9.26}, by the definition of $\eta$,
\begin{align*}
\int_{B_\rho(x_0)}\eta(x_0-y)\overline{u}(s,y)\,dy=\int_{B_\rho(x_0)} \eta(x_0-y)u(s,y)\,dy.
\end{align*}
Therefore,
\begin{align*}
&|\int_{B_\rho(x_0)}\eta(x_0-y)\overline{u}(t,y)\,dy-\int_{B_\rho(x_0)}\eta(x_0-y)\overline{u}(s,y)\,dy|\\&=|\int_{B_\rho(x_0)}\eta(x_0-y)u(t,y)\,dy-\int_{B_\rho(x_0)}\eta(x_0-y)u(s,y)\,dy|\\
&=|\int_s^t \int_{B_\rho(x_0)}\eta(x_0-y)u_t(\tau,y)\,dy\,d\tau|.
\end{align*}
Because $u$ satisfies the equation,
\begin{align}\nonumber
&|\int_s^t \int_{B_\rho(x_0)}\eta(x_0-y)u_t(\tau,y)\,dy\,d\tau|\\&=|\int_s^t \int_{B_\rho(x_0)}\eta(x_0-y)((-1)^{m+1}\sum_{|\alpha|\le m, |\beta|\le m}D^\alpha (A^{\alpha\beta}D^\beta u(\tau,y))\nonumber\\
&\,\,+\sum _{|\alpha|\le m} D^\alpha f_\alpha)\,dy\,d\tau|\label{ut estimate}.
\end{align}
Since $\eta(x_0-y)$ has compact support in $B_\rho(x_0)$, we use $f_\alpha -(f_\alpha)_{B_\rho(x_0)}$ to substitute $f_\alpha$ when $|\alpha|=m$.
Moreover for $|\alpha|=m$,
\begin{align*}
&D^\alpha (A^{\alpha\beta}D^\beta u)\\
&=D^\alpha \big((A^{\alpha\beta}-(A^{\alpha\beta})_{B_\rho(x_0)})D^\beta u\big)+D^\alpha( (A^{\alpha\beta})_{B_\rho(x_0)}(D^\beta u-(D^\beta u)_{B_\rho(x_0)})).
\end{align*}
We plug all these into \eqref{ut estimate} and integrate by parts. It follows easily that
\begin{align*}
&|\int_s^t \int_{B_\rho(x_0)}\eta(x_0-y)u_t(\tau,y)\,dy\,d\tau|\\
&\le \sum_{|\alpha|=m, |\beta|\le m}([A^{\alpha\beta}]_a^\ast\|D^\beta u\|_{L_\infty(Q_{2R})}+{[D^\beta u]_{a,Q_{2R}}^*}\|A^{\alpha\beta}\|_{L_\infty})\rho^a\\
&\,\,\cdot\int_{Q_\rho(X_0)}|D^m\eta(x_0-y)|\,dy\,dt+\sum_{|\alpha|=m}\|D^m \eta\|_{L_\infty}[f_\alpha]_{a,Q_{2R}}^\ast|Q_\rho|\rho^a \\
&\,\,+\sum_{|\alpha|<m, |\beta|\le m}\|A^{\alpha\beta}\|_{L_\infty}\|D^\beta u\|_{L_\infty(Q_{2R})}\int_{Q_{\rho}(X_0)}|D^\alpha \eta(x_0-y)|\,dy\,dt\\
&\,\,+\sum_{|\alpha|<m}\|f_\alpha\|_{L_\infty(Q_{2R})}\int_{Q_\rho(X_0)}|D^\alpha \eta(x_0-y)|\,dy\,dt\\
&\le C(\|D^m u\|_{\frac{a}{2m},a,Q_{2R}}+{\sum_{|\beta|<m}{\|D^\beta u\|_{a, Q_{2R}}^*}})\rho^{m+a}+C\sum_{|\alpha|=m}[f_\alpha]_{a,Q_{2R}}^\ast\rho^{m+a}\\&
\,\,+C\sum_{|\alpha|<m}(\|f_\alpha\|_{L_\infty(Q_{2R})}+\sum_{|\beta|\le m}\|A^{\alpha\beta}\|_{L_\infty}\|D^\beta u\|_{L_\infty(Q_{2R})})\rho^{m+1}.
\end{align*}
Here we used $\|D^\alpha \eta\|_{L_\infty}\le C \rho^{-d-|\alpha|}$ for any $|\alpha|\le m$.
Hence, {combining with} \eqref{eq4.14},
\begin{align*}
|u(t,x_0)-u(s,x_0)|\le C\rho^{m+a}(\|D^m u\|_{\frac{a}{2m},a,Q_{2R}}+\sum_{|\beta|\le m}\|D^\beta u\|_{L_\infty(Q_{2R})}+F),
\end{align*}
which together with \eqref{eq4.42} implies \eqref{eq4.43}. This completes the proof of the proposition.

\subsection{Boundary estimates}
As in the previous subsection, we first consider systems with simple coefficients without lower-order terms:
\begin{align}
&u_t+(-1)^m\l_0u=0 \quad \text{in}\quad Q^+_{2R},\label{beq}\\
&u=0,~D_du=0,~\cdots,~ D_d^{m-1}u=0 \quad \text{on} \quad \{x_d=0\}\cap Q_{2R}.\label{bcd}
\end{align}

Thanks to the $L_p$ estimates in the half space obtained in \cite{DK10}, we are able to derive a local $W_p^{1,2m}$ estimate for solution of \eqref{beq} and \eqref{bcd} by the method in Lemma 4.1 in \cite{DK10}. Namely,
\begin{equation}
\label{lp}\|u\|_{W_p^{1,2m}(Q_R^+)}\le C(R,p)\|u\|_{L_p(Q^+_{2R})},
\end{equation}
where $C(R,p)$ is a constant.
Now we are ready to prove the following lemma.

\begin{lemma}\label{keybd}
Assume that $0<r\le R<\infty$, $\gamma \in (0,1)$, and $u\in C_{loc}^\infty({\overline{{\mathbb{R}_+^{d+1}}}})$ satisfies \eqref{beq} and \eqref{bcd}. Then there exists a constant $C=C(d,n,m,\lambda,K, \gamma)$, such that
\begin{align}
&\int_{Q^+_r(X_0)}|u|^2\,dx\,dt\le C(\frac{r}{R})^{d+4m}\int_{Q^+_R(X_0)}|u|^2\,dx\,dt \label{morrey},\\
&\int_{Q^+_r(X_0)}|D^l_d u|^2 \,dx\,dt \le C (\frac{r}{R})^{2(m-l)+d+2m}\int_{Q^+_R(X_0)}|D^l_d u|^2 \,dx\,dt,\label{lownor}\\
&\nonumber\int_{Q^+_r(X_0)}|D_d^mu-(D_d^mu)_{Q^+_r(X_0)}|^2\,dx\,dt \\
&\le C(\frac{r}{R})^{2\gamma+d+2m}\int_{Q^+_R(X_0)}|D_d^mu-(D_d^mu)_{Q_R^+(X_0)}|^2\,dx\,dt,\label{com}
\end{align}
where $X_0 \in \{x_d=0\}\cap Q_R$  and  ${0}<l<m$.
\end{lemma}
\begin{proof}
By scaling and translation of the coordinates, without loss of generality, we can assume $R=1$ and $X_0=(0,0)$. From \eqref{lp},  for any $r_1<2$, there exists a constant $C=C(r_1)$ so that
\begin{equation*}
\label{boot2}\|u\|_{W_2^{1,2m}(Q^+_{r_1})}\le C\|u\|_{L_2(Q^+_{2})}.
\end{equation*}
By the Sobolev embedding theorem,
\begin{equation*}
\|u\|_{L_p(Q^+_{r_1})}\le C \|u\|_{W_2^{1,2m}(Q^+_{r_1})},
\end{equation*}
where $1/p>1/2-1/(d+1)$. Using \eqref{lp} again, we get
\begin{equation*}
\|u\|_{W_p^{1,2m}(Q^+_{r_2})}\le C\|u\|_{L_p(Q^+_{r_1})} \le C \|u\|_{L_2(Q^+_{2})},
\end{equation*}
where $r_2<r_1$ and $C=C(r_1,r_2)$.
Due to a standard argument of bootstrap with a sequence of shrinking cylinders$ ~Q_{r_l}$, for any $p>2$ there is a $C$ depends on $p$ such that
\begin{equation}\label{boot}
\|u\|_{W_p^{1,2m}(Q^+_1)}\le C\|u\|_{L_2(Q^+_{2})}.
\end{equation}
By the Sobolev embedding theorem, we have
\begin{align}
&W_p^{1,2m}(Q^+_1)\hookrightarrow C^{\frac{a}{2m},a}(Q^+_1),\nonumber\\
&\|u\|_{\frac{a}{2m},a,Q_1^+}\le C\|u\|_{W^{1,2m}_p(Q^+_1)},\label{embed ineq}
\end{align}
where $a=2m-(d+2m)/p$.
Since $p<\infty$ can be arbitrarily large, $a $ can be arbitrarily close to $2m$ from below.

Let us first prove (\ref{morrey}). We only need to consider $r< 1/2$, because otherwise it is obvious. By the Poincar\'e inequality and \eqref{embed ineq} with some $a>m$,
\begin{align*}
\int_{Q^+_r}|u|^2\,dx\,dt&\le C r^{2m}\|D_d^mu\|_{L_\infty(Q^+_r)}^2|Q^+_r|\le C r^{d+4m}\|D_d^mu\|_{L_\infty(Q^+_{1/2})}^2\\
&\le Cr^{d+4m}\|u\|^2_{W_p^{1,2m}(Q^+_{1/2})} \le C r^{d+4m}\int_{Q^+_1}|u|^2\,dx\,dt.
\end{align*}
The last inequality is due to the $L_p$ estimate in \eqref{boot} with $1/2$ and $1$ in place of $1$ and $2$ respectively.
Next, we prove \eqref{lownor} and only consider the case when $r<1/2$ as in proving \eqref{morrey}. By the Poincar\'e inequality and \eqref{embed ineq} with some $a>m$, we have
\begin{align*}
&\int_{Q^+_r} |D^l_d u|^2 \,dx\,dt \le r^{2(m-l)} \int_{Q^+_r}|D^m_d u|^2 \,dx\,dt\\
&\le Cr^{2(m-l)+d+2m} \sup_{Q^+_{1/2}} |D^m_d u|^2 \le Cr^{2(m-l)+d+2m}\|u\|^2_{W_p^{1,2m}(Q^+_{1/2})}\\
&\le C r^{2(m-l)+d+2m}\int_{Q^+_1}|u|^2\,dx\,dt
\le C r^{2(m-l)+d+2m}\int_{Q^+_1}|D^l_du|^2\,dx\,dt.
\end{align*}
As for (\ref{com}), let $v=u-\frac{x_d^m}{m!}(D^m_du)_{Q^+_1}$ so that $v$ also satisfies the system as $u$. Moreover,
$D_d^mv=D_d^mu-(D_d^mu)_{Q^+_1}$. Then, by  \eqref{embed ineq}, \eqref{boot}, and the Poincar\'e inequality, for $a={m+\gamma}\in(m,m+1)$ and $r<1/2$,
\begin{multline*}
\int_{Q^+_r}|D_d^mv-(D_d^mv)_{Q^+_r}|^2\,dx\,dt\le C r^{2(a-m)+d+2m}\|v\|^2_{\frac{a}{2m},a,Q_r^+}\\
 \le C r^{2{\gamma}+d+2m} \int_{Q^+_1}|v|^2\,dx\,dt\le  Cr^{2{\gamma}+d+2m}\int_{Q^+_1}|D_d^mv|^2\,dx\,dt.
\end{multline*}
The lemma is proved.
\end{proof}
Similar to Lemma \ref{divsim}, we consider
\begin{equation}
\label{bddiv}u_t+(-1)^m\l_0u=\sum_{|\alpha|\le m}D^\alpha f_{\alpha} \quad \text{in} \quad Q^+_{2R},
\end{equation}
and $u$ satisfies (\ref{bcd}).

\begin{lemma} \label{bdholder}
Assume that $u\in C_{loc}^\infty({\overline{\bR^{d+1}_+}})$ satisfies \eqref{bddiv} and \eqref{bcd}, and  $f_\alpha \in C^{a \ast}$ for $|\alpha| =m$ where $a\in (0,1)$, $f_\alpha \in L_\infty$ for $|\alpha|< m$.
Then for any $0<r<R\le1, \gamma\in (0,1)$ and $X_0\in \{x_d=0\}\cap Q_R$, there exists a constant $C=C(d,m,n,\lambda, K, \gamma)$ such that
\begin{align}
&\nonumber\int_{Q^+_r(X_0)}|D^m_{x^\prime}u|^2 \,dx\,dt\\&\le C (\frac{r}{R})^{4m+d} \int_{Q^+_R(X_0)}|D^m_{x^\prime}u|^2\,dx\,dt+C\sum_{|\alpha|=m}[f_\alpha]_{a,Q_{2R}^+} ^{\ast 2}R^{2a+2m+d} \nonumber
\\&\,\,+C\sum_{|\alpha|<m}\|f_\alpha\|^2_{L_\infty(Q_{2R}^+)}R^{2+2m+d}\label{bdmor},
\end{align}
\begin{align}
&\nonumber\int_{Q^+_r(X_0)}|D^m_d u-(D^m_du)_{Q^+_r(X_0)}|^2\,dx\,dt\\
&\le \nonumber C(\frac{r}{R})^{2\gamma+d+2m}\int_{Q^+_R(X_0)}|D^m_d u-(D^m_du)_{Q^+_R(X_0)}|^2\,dx\,dt \\
&\,\,+C\sum_{|\alpha|=m}[f_\alpha]_{a,Q_{2R}^+} ^{\ast 2}R^{2a+2m+d}+C\sum_{|\alpha|<m}\|f_\alpha\|^2_{L_\infty(Q_{2R}^+)}R^{2+2m+d}\label{bdcom},
\end{align}
\begin{align}
&\nonumber\int_{Q^+_r(X_0)}|D_d^lD_{x^\prime}^{m-l}u|^2 \,dx\,dt\\&\le C(\frac{r}{R})^{2(m-l)+d+2m} \int_{Q^+_R(X_0)}|D_d^lD_{x^\prime}^{m-l}u|^2\,dx\,dt
\nonumber\\&\,\,+ C\sum_{|\alpha|=m}[f_\alpha]_{a,Q_{2R}^+} ^{\ast 2}R^{2a+2m+d} +C\sum_{|\alpha|<m}\|f_\alpha\|^2_{L_\infty(Q_{2R}^+)}R^{2+2m+d}\label{bdmix},
\end{align}
where $0<l<m$.
\end{lemma}
\begin{proof}
We follow the proof of Lemma \ref{divsim}. Let $w$ be the weak solution of the following system
$$
\left\{\begin{aligned}
&w_t+(-1)^m\l_0w=\sum_{|\alpha|\le m}D^\alpha f_\alpha \quad \text{in}\quad Q^+_{2R},\\
&w=0,~|D_dw|=0,\cdots,~|D^{m-1}_dw|=0 \quad \text{on} \quad \partial_pQ^+_{2R}.
\end{aligned}\right.
$$
Then the Poincar\'e inequality and the method in the proof of Lemma \ref{divsim} yield
\begin{align*}
&\int_{Q^+_{2R}}|D^mw|^2\,dx\,dt\\
&\le C\sum_{|\alpha|=m}[f_\alpha]_{a,Q_{2R}^+}^{\ast2} R^{2a+2m+d} +C\sum_{|\alpha|<m}\|f_\alpha\|^2_{L_\infty(Q_{2R}^+)}R^{2+2m+d}.
\end{align*}
We again use the mollification argument as in Lemma \ref{divsim} so that $w$ is smooth. Let $v=u-w$. Then $v$ satisfies all the conditions in Lemma \ref{keybd}, and so does $D_{x^\prime}^m v$. Thus we obtain \eqref{morrey} with $D_{x^\prime}^mv$ in place of $u$.

Combining the estimates of $v$ and $w$, we obtain \eqref{bdmor}. The proofs of \eqref{bdcom} and \eqref{bdmix} are similar.
\end{proof}
Now we are ready to prove the H\"{o}lder estimates similar to the interior case. From \eqref{bdmor}, \eqref{bdcom}, and \eqref{bdmix}, taking $\gamma>a$ and using Lemma \ref{tech}, we get for any $X_0\in \{x_d=0\}\cap Q_R$ and $r\in (0,R)$,
\begin{equation*}
\int_{Q^+_r(X_0)}|D^m_{x^\prime}u|^2+|D^m_d u-(D^m_du)_{Q^+_r(X_0)}|^2+|D_d^lD_{x^\prime}^{m-l}u|^2 \,dx\,dt\le I \end{equation*}
where
\begin{align*}
&I:=C( R^{-(2a+2m+d)} \int_{Q^+_{2R}}|D^mu|^2\,dx\,dt+F^2)r^{2a+2m+d}\\ &F=\sum_{|\alpha|=m}[f_\alpha]_{a,Q_{2R}^+}^\ast+\sum_{|\alpha|<m}\|f_\alpha\|_{L_\infty(Q_{2R}^+)}. \end{align*}
This estimate, together with the interior estimates and Lemma \ref{holder}, implies that
\begin{equation*}
[D^mu]_{\frac{a}{2m},a,Q^+_R}\le C(R^{-(2a+2m+{d})}\int_{Q^+_{2R}}|D^m u|^2\,dx\,dt+F^2)^{\frac{1}{2}}.
\end{equation*}
Similar to  what we did in proving \eqref{eq4.42}, we can apply the argument of freezing the coefficients to deal with the variable coefficients case with lower-order terms. Here we just state the conclusion:
\begin{proposition}
Assume that $R\le 1$ and $u\in C_{loc}^\infty({\overline{\bR^{d+1}_+}})$ satisfies
$$
\left\{\begin{aligned}
&u_t+(-1)^m \l u=\sum_{|\alpha|\le m}D^\alpha f_\alpha \quad \text{in}\quad Q_{2R}^+,\\
&u=0,\,D_du=0,\,\cdots, \,D_d^{m-1}u=0\quad \text{on}\quad \{x_d=0\}\cap Q_{2R}.
\end{aligned}\right.
$$
Suppose that $\l$ and $f_\alpha$ satisfy the conditions in Theorem \ref{existence}. Then there exists a constant $C=C(d,n,m,\lambda, K, \|A^{\alpha\beta}\|_a^\ast, R, a)$ such that
\begin{equation*} 
\langle u\rangle_{\frac{1}{2}+\frac{a}{2m}, Q_R^+}+[D^m u]_{\frac{a}{2m},a,Q^+_R}\le C(\sum_{|\beta|\le m}\|D^{\beta}u\|_{L_\infty(Q^+_{2R})}+F).
\end{equation*}
\end{proposition}

\mysection{Proof of Theorem \ref{existence}}
Before proving the main theorem, let us show a technical lemma.
\begin{lemma}\label{tech2}
Assume that $a\in(0,1)$, $u \in C^\infty_{loc}(\mathbb{R}^{d+1})$, $\Omega $ is an open set in $\mathbb{R}^d$ and $\partial \Omega \in C^2$, $T\in (0,\infty]$, then for any $\varepsilon>0$
\begin{equation*}
\|u\|_{L_\infty((0,T)\times \Omega)} \le \varepsilon [u]_{\frac{a}{2m},a,(0,T)\times \Omega}+C(\varepsilon) \|u\|_{L_2((0,T)\times \Omega)},
\end{equation*}
where $C(\varepsilon)$ is a constant depending on $\varepsilon$.
\end{lemma}
\begin{proof}
Since $\partial\Omega\in C^2$, there exists a $r_0>0$ such that for any $X_0 \in (0,T)\times \Omega$, we can choose an arbitrarily small cylinder $Q_r(s,y) \subset \Omega$ with $r\in (0, r_0)$ so that $X_0 \in Q_r(s,y)$ and
\begin{align*}
|u(X_0)|&\le|u(X_0)-\dashint_{Q_r(s,y)}u(t,x) \,dx\,dt|+|\dashint_{Q_r(s,y)}u(t,x) \,dx\,dt|\\
& \le Cr^{a}[u]_{\frac{a}{2m},a}+Cr^{-\frac{d+2m}{2}}\|u\|_{L_2}.
\end{align*}
Therefore,
\begin{equation*}
\|u\|_{L_\infty} \le  \varepsilon[u]_{\frac{a}{2m},a}+C(\varepsilon)\|u\|_{L_2}.
\end{equation*}
The lemma is proved.
\end{proof}
Next, we show a global a priori estimate.
\begin{proposition}
Let $\l$, $\Omega$, and $f_\alpha$ satisfy the conditions in Theorem \ref{existence}.
Assume that $u\in {C^\infty((0,T)\times\Omega)}$ and satisfies the following system
\begin{equation}        \label{cylindereq}
\left\{\begin{aligned}
& u_t+(-1)^m\l u=\sum_{|\alpha|\le m}D^\alpha f_\alpha \quad \text{in} \quad (0,T)\times \Omega,\\
&u=0,\,|D u|=0,\,\cdots,\,|D^{m-1}u|=0 \quad \text{on} \quad   [0,T)\times \partial\Omega,\\
& u=0 \quad \text{on} \quad\{0\}\times \Omega.
\end{aligned}\right.
\end{equation}
Then  \eqref{finaldiv} holds with $g=0$.
\end{proposition}
\begin{proof}
Using  the standard arguments of partition of unity and flattening the boundary, we combine the interior, boundary estimates, and the estimate in $t$ variable for the divergence form systems to get
\begin{equation*}
\langle u\rangle_{\frac{a+m}{2m},(0,T)\times \Omega}+[D^mu]_{\frac{a}{2m},a,(0,T)\times \Omega}\le C(\sum_{|\beta|\le m}\|D^{\beta}u\|_{L_\infty((0,T)\times \Omega)}+F).
\end{equation*}
By the interpolation inequalities in H\"{o}lder spaces, for instance, see \cite[Section 8.8]{Kry96},
\begin{equation*}
\langle u\rangle_{\frac{a+m}{2m},(0,T)\times \Omega}+[D^mu]_{\frac{a}{2m},a,(0,T)\times \Omega} \le C(\|u\|_{L_\infty((0,T)\times \Omega)}+F).
\end{equation*}
Applying Lemma \ref{tech2} and the interpolation inequalities again, we get
\begin{equation*}
\|u\|_{L_\infty((0,T)\times \Omega)} \le C(\varepsilon) \|u\|_{L_2((0,T)\times \Omega)}+\varepsilon [D^m u]_{\frac{a}{2m},a,(0,T)\times \Omega}{+\varepsilon \langle u\rangle_{\frac{a+m}{2m},(0,T)\times \Omega}}.
\end{equation*}
Upon taking $\varepsilon$ sufficiently small, we arrive at  \eqref{finaldiv}.
\end{proof}
In order to implement the method of continuity, we need the right-hand side of \eqref{finaldiv} to be independent of $u$ and this leads us to consider the following system,
\begin{equation}     \label{exist}
\left\{\begin{aligned}
&u_t+(-1)^m\l u+\kappa u=\sum_{|\alpha|\le m}D^\alpha f_\alpha\quad \text{in}\quad (0,T)\times \Omega,\\
&u=0,~|Du|=0,\cdots,|D^{m-1}u|=0\quad \text{on}\quad [0,T)\times \partial\Omega, \\
& u=0\quad\text{on} \quad \{0\}\times \Omega,
\end{aligned}\right.
\end{equation}
where $\kappa$ is a large constant to be specified later.
We rewrite the system as
\begin{equation}\label{reexist}
u_t+(-1)^m\l_h u+\kappa u=\sum_{|\alpha|\le m}D^\alpha f_\alpha +(-1)^m(\l_h-\l)u,
\end{equation}
where $\l_h u$ denotes the sum of the highest-order terms.
Then multiply $u$ to both sides of \eqref{reexist} and integrate over $(0,T)\times \Omega$. Thus
\begin{align}
&\|D^m u\|^2_{L_2((0,T)\times \Omega)} +\kappa \|u\|^2_{L_2((0,T)\times \Omega)}\nonumber \\&\le \sum_{|\alpha|\le m}(-1)^{|\alpha|}\int_{(0,T)\times \Omega} f_\alpha D^\alpha u\,dx\,dt\nonumber\\
&\,\,+\sum_{|\alpha|<m, |\beta|\le m}(-1)^{m+|\alpha|+1}\int_{(0,T)\times \Omega}A^{\alpha\beta}D^\beta u D^\alpha u \,dx\,dt\nonumber\\&\,\,-\sum_{|\alpha|=m,|\beta|<m} \int_{(0,T)\times \Omega}A^{\alpha\beta}D^\beta u D^\alpha u \,dx\,dt.\label{k estimate}
\end{align}
Take a point $x_0\in \Omega$. Due to the homogeneous Dirichlet boundary condition, if $|\alpha|=m$, the factor $f_\alpha$ in the first integral on the right-hand side above can be replaced by $f_\alpha(t,x)-f_\alpha(t,x_0)$.
We use the Cauchy--Schwarz inequality, Young's inequality, and the interpolation inequality to bound the right-hand side by
\begin{multline*}
C(n,m,d,\lambda,K,\|A^{\alpha\beta}\|_a^\ast, \varepsilon) \Big(\sum_{|\alpha|\le m}\int_{(0,T)\times \Omega} |f_\alpha-f_\alpha(t,x_0)1_{|\alpha|=m}|^2\,dx\,dt\\
+\int_{(0,T)\times \Omega}|u|^2 \,dx\,dt\Big)+\varepsilon\|D^m u\|^2_{L_2((0,T)\times \Omega)}.
\end{multline*}
After taking $\varepsilon$ sufficiently small to absorb the term $\varepsilon\|D^m u\|^2_{L_2((0,T)\times \Omega)}$ to the left-hand side of \eqref{k estimate} and
choosing $\kappa $ sufficiently large, we reach
\begin{equation}\label{l2control}
\|u\|_{L_2((0,T)\times \Omega)} \le CF,
\end{equation}
where $C=C(d,n,m,\lambda,K, \|A^{\alpha\beta}\|_a^\ast,\Omega)$.
Combining \eqref{l2control} with \eqref{finaldiv}, we get the following lemma.
\begin{lemma}                   \label{lem5.3}
Assume that $\l$, $\Omega$, and $f_\alpha$ satisfy all the conditions in Theorem \ref{existence} and $u \in {C^\infty((0,T)\times\Omega)} $ satisfies \eqref{exist} with a sufficiently large constant $\kappa$ depending on $n,m,d,\lambda,K$ and $\|A^{\alpha\beta}\|_a^\ast$,
then
\begin{equation}\label{finexist}
\|u\|_{\frac{a+m}{2m}, a+m,{(0,T)\times \Omega}} \le C F,
\end{equation}
where $C=C(d,m,n,\lambda,K,\|A^{\alpha\beta}\|_a^\ast, \Omega, a)$.
\end{lemma}
Now we are ready to prove Theorem \ref{existence}.
\begin{proof}[Proof of Theorem \ref{existence}]
By considering $u-g$ instead of $u$, without loss of generality, we can assume $g=0$. It remains to prove the solvability. Let $\kappa$ be the constant from the previous lemma. Since $u$ satisfies \eqref{cylindereq}, the function $v:=e^{-\kappa t}u$ satisfies
\begin{equation*}
v_t+(-1)^m\l v+\kappa v=e^{-\kappa t}\sum_{|\alpha|\le m}D^\alpha f_\alpha.
\end{equation*}
We then reduce the problem to the solvability of $v$.
By Lemma \ref{lem5.3}, \eqref{finexist} holds for $v$. Consider the following equation \begin{equation*}
 v_t+(-1)^m(s\l v+(1-s)\Delta ^{m} v)+\kappa v=e^{-\kappa t}\sum_{|\alpha|\le m}D^\alpha f_\alpha
\end{equation*}
with the same initial and boundary conditions,
where the parameter $s \in [0,1]$.  It is known that when $s=0$ there is a unique solution in $C^{\frac{a+m}{2m},a+m}([0,T)\times\Omega)$. Then by the method of continuity  and the a priori estimate \eqref{finexist}, we find a solution when $s=1$ which gives $u$.
 \end{proof}

\begin{remark}
Actually the condition $f_\alpha \in L_\infty$ for $|\alpha|<m$ can be relaxed. With slight modification in our proof, $f_\alpha$ can be in some Morrey spaces.
\end{remark}

\mysection{Estimate for non-divergence type systems}
In this section, we deal with the non-divergence type systems.
\subsection{Interior estimate}
First, we consider the interior estimates for  the non-divergence form system
\begin{equation*}
u_t+(-1)^mL_0u=f \quad in\quad Q_{2R}.
\end{equation*}
For this system, the interior estimates are simple consequences of the corresponding estimates for divergence form systems.
We differentiate the system $m$ times with respect to $x$ to get
\begin{equation*}
D^mu_t+(-1)^mL_0D^mu=D^mf.
\end{equation*}
Thanks to the estimates of the divergence type systems, by \eqref{fixeq},
\begin{equation*}
[D^{2m}u]_{\frac{a}{2m},a,Q_R}\le C(\|D^{2m}u\|_{L_\infty(Q_{2R})}+[f]_{a,Q_{2R}}^{\ast}).
\end{equation*}
Next we deal with the general non-divergence form systems. The coefficients are all in $C^{a\ast }$.  Following exactly the same idea as in handling the divergence form, we first consider $L=\sum_{|\alpha|=|\beta|=m}A^{\alpha\beta}D^\alpha D^\beta$.  Fix $y\in B_R^+$ and let
$$
L_{0y}=\sum_{|\alpha|=|\beta|=m}A^{\alpha\beta}(t,y)D^\alpha D^\beta.
$$
We rewrite the system as follows
\begin{equation*}
u_t+(-1)^m L_{0y} u= f+(-1)^m(L_{0y}-L)u.
\end{equation*}
Set $\zeta$ to be an infinitely differentiable function in $\mathbb{R}^{d+1}$ such that
\begin{equation}
\zeta =1 \quad\text{in}\quad Q_{2R}, \quad \zeta =0 \quad\text{outside}\quad (-(4R)^{2m}, (4R)^{2m})\times B_{4R}\label{zeta1}.
\end{equation}
From Theorem 2.6 in \cite{DK10}, for $T=(4R)^{2m}$, let $w$ be the unique $W_2^{1,2m}$-solution of the following system
\begin{equation*}
w_t+(-1)^mL_{0y}w=\zeta(f-(f)_{B_{4R}(t)})+\zeta(-1)^m(L_{0y}-L)u
\end{equation*}
in $(-T,0)\times \bR^{d}$ {with the zero initial condition at $t=-T$}. It is easy to obtain the estimate for $w$,
\begin{align*}
&\|D^{2m}w\|^2_{L_2(Q_{4R})} \le C\|f-(f)_{B_{4R}}(t)\|^2_{L_2(Q_{4R})}
+C\|(L_{0y}-L)u\|^2_{L_2(Q_{4R})}\nonumber\\
&\,{\le CR^{2a+2m+d}[f]^{\ast2}_{a,Q_{4R}}+ C R^{2a+2m+d}\sum_{|\alpha|=|\beta|=m}[A^{\alpha\beta}]_a^{\ast2}\|D^{2m} u\|^2_{L_\infty(Q_{4R})}}.
\end{align*}
We apply the mollification argument so that $w$ is smooth. Let $v=u-w$ which is also a smooth function. Moreover since $\zeta=1$ in $Q_{2R}$, $D^{2m} v$ satisfies \eqref{hom}. Therefore \eqref{eq 4.08} holds for $D^{2m}v$. Combining the estimate of $w$ and $D^{2m}v$, similar to the proof of Proposition \ref{div_complete}, we get
\begin{equation}
[D^{2m} u]_{\frac{a}{2m},a,Q_R}\le C(\|D^{2m}u\|_{L_\infty(Q_{4R})}+[f]_{a,Q_{4R}}^\ast)\label{eq6.40}.
\end{equation}
For the operator $L$ with lower-order terms, we rewrite the systems as follows
\begin{equation*}
u_t+(-1)^mL_h u=f+(-1)^{m+1}\sum_{|\alpha|+|\beta|<2m}A^{\alpha\beta}D^\alpha D^\beta u.
\end{equation*}
where $L_h=\sum_{|\alpha|=|\beta|=m}A^{\alpha\beta}D^\alpha D^\beta$.
Let
$$
\tilde{f}=f+(-1)^{m+1}\sum_{|\alpha|+|\beta|<2m}A^{\alpha\beta}D^\alpha D^\beta u.
$$
An easy calculation shows that
\begin{align*}
[\tilde{f}]_{a, Q_{4R}}^\ast&\le C\big([f]_{a, Q_{4R}}^\ast+\sum_{|\alpha|+|\beta|<2m}(\|A^{\alpha\beta}\|_{L_\infty}{[D^\alpha D^\beta u]_{a, Q_{4R}}^*}\\
&\,+[A^{\alpha\beta}]_a^\ast\|D^\alpha D^\beta u\|_{L_\infty(Q_{4R})})\big).
\end{align*}
From \eqref{eq6.40} with $\tilde{f}$ in place of $f$, the following estimate holds
\begin{equation*}
[D^{2m}u]_{\frac{a}{2m},a,Q_R}\le C(\sum_{|\gamma|\le2m}\|D^\gamma u\|_{L_\infty(Q_{4R})}+[f]_{a,Q_{4R}}^{\ast }),
\end{equation*}
where $C=C(d,m,n,\lambda,K,\|A^{\alpha\beta}\|_a^\ast, R, a)$.
Because
\begin{equation*}
u_t=(-1)^{m+1}Lu+f,
\end{equation*}
it follows immediately that
\begin{equation*}
\|u_t\|_{a, Q_R}^*+[D^{2m}u]_{\frac{a}{2m},a,Q_R}\le C(\sum_{|\gamma|\le2m}\|D^\gamma u\|_{L_\infty(Q_{4R})}+\|f\|_{a,Q_{4R}}^{\ast}).
\end{equation*}
Implementing a standard interpolation argument, {for instance, see  \cite[Section 8.8]{Kry96} and Lemma 5.1 of \cite{Gia93}}, we are able to prove the first part of Theorem \ref{nondiv theorem}.

\subsection{Boundary estimate} In the non-divergence case, better estimates than the ones in Lemma \ref{keybd} are necessary.
For \eqref{com}, we actually can estimate more normal derivatives up to $2m-1$-th order.  In order to show this, let us prove the following lemma. A similar result can be found in \cite{DK11}.
 \begin{lemma}\label{tech3}
 Assume that $u\in C_{loc}^\infty({\overline{\bR^{d+1}_+}})$ satisfies \eqref{beq} and \eqref{bcd}. Let $Q(x)$ be a vector-valued polynomial of order $m-1$ and $P(x)=x_d^m Q(x)$. Suppose that
 \begin{equation*}
 (D^k D_d^m P(x))_{Q^+_R}=(D^k D_d^m u(t,x))_{Q^+_R},
 \end{equation*}
where $0\le k\le m-1$. Let $v=u-P(x)$, then there exists a constant $C=C(d,m,n,\lambda)$ such that
\begin{equation*}
\|D^k D_d^mv\|_{L_2(Q^+_R)} \le CR^{m-1-k}\|D^{m-1}D_d^m v\|_{L_2(Q^+_R)}
\end{equation*}
for any $0\le k\le m-1$.
 \end{lemma}
 \begin{proof}
 Without loss of generality, let us assume $R=1$. Choose $\xi(y)\in C^\infty_0(B^+_1)$ with unit integral. Then let
 \begin{equation*}
 g_k(t)=\int_{B^+_1} \xi(y) D^k D_d^mv(t,y)\,dy,\quad c_k=\int_{-1}^0 g_k(t)dt.
 \end{equation*}
 By the Poincar\'e inequality and H\"{o}lder's inequality, the following estimate holds
 \begin{align*}
& \int_{B^+_1}|D^k D_d^mv(t,x) -g_k(t)|^2 \,dx\\
&=\int_{B^+_1}|\int_{B^+_1}(D^k D_d^mv(t,x)-D^k D_d^mv(t,y))\xi(y)\,dy|^2dx\\
&\le C\int_{B^+_1}\int_{B^+_1}|(D^k D_d^mv(t,x)-D^k D_d^mv(t,y))|^2\,dy\,dx\\
&\le C\int_{B^+_1}|D^{k+1}D_d^m v(t,y)|^2\,dy.
 \end{align*}
 Because $(D^k D_d^mv)_{Q^+_1}=0$, we have
 \begin{align}\nonumber
 &\|D^k D_d^mv\|_{L_2(Q^+_1)}\le \|D^k D_d^mv-c_k\|_{L_2(Q^+_1)}\\
 &\le \|D^k D_d^mv-g_k(t)\|_{L_2(Q^+_1)}+\|g_k(t)-c_k\|_{L_2(Q^+_1)}\nonumber\\
 &\le C\|D^{k+1} D_d^mv\|_{L_2(Q^+_1)}+C\|\partial_t g_k\|_{L_2(-1,0)}.\label{tech3.2}
 \end{align}
Since $v$ satisfies the same system as $u$, by the definition of $g_k$ and integrating by parts
\begin{align*}
\partial_tg_k(t)&=\int_{B^+_1}\xi(y)D^k D_d^m\partial_tv(t,y)\,dy\\
&=(-1)^{m+k+1}\int_{B^+_1}D^k \xi(y)\l_0D^m_dv(t,y)\,dy.
\end{align*}
We integrate by parts again leaving $m-1$ derivatives on $D^m_d v$ and moving all the others onto $\xi$ to get
 \begin{align*}
 {|\partial_tg_k(t)|}\le C\int_{B_1^+}|D^{m-1}D^m_d v|\,dy.
 \end{align*}
Therefore,
 \begin{equation}
 |\partial_tg_k(t)|^2\le C\|D^{m-1}D_d^m v(t,\cdot)\|^2_{L_2(B^+_1)}.\label{tech3.3}
 \end{equation}
Combining \eqref{tech3.2} and \eqref{tech3.3}, we prove the lemma by induction.
 \end{proof}

Next, we prove an estimate for $D^{2m-1} u$.
\begin{lemma}\label{nordirection}
Assume that $u\in C_{loc}^{\infty}(\overline{{\bR_+^{d+1}}})$ and satisfies \eqref{beq} and \eqref{bcd}. Then for any $0<r\le R<\infty$ and $\gamma \in (0,1)$, there exists a constant $C=C(d,n,m,\lambda,\gamma)$ such that for any $X_0\in \{x_d=0\}\cap Q_R$
\begin{align*}
&\int_{Q^+_r(X_0)}|D^{2m-1}u-(D^{2m-1}u)_{Q_r^+(X_0)}|^2\,dx\,dt\\&\le C(\frac{r}{R})^{2\gamma+d+2m}\int_{Q^+_R(X_0)}|D^{2m-1}u-(D^{2m-1}u)_{Q_R^+(X_0)}|^2\,dx\,dt.
\end{align*}
\end{lemma}
\begin{proof}
By scaling and translation of the coordinates, without loss of generality, we can assume $R=1$  and $X_0=(0,0)$. Moreover, we only need to consider the case when $r<1/2$ for the same reason as in the proof of Lemmas \ref{fund} and \ref{keybd}.  Recall that in Lemma \ref{keybd} we have
\begin{equation*}
\|u\|_{\frac{a}{2m},a,Q_1^+}\le C\|u\|_{W^{1,2m}_p(Q^+_1)}\le C\|u\|_{L_2(Q_2^+)},
\end{equation*}
where $a<2m$ can be arbitrarily close to $2m$. Set $a{=\gamma+2m-1}\in(2m-1, 2m)$. Applying the inequality above with $1/2$ and 1 in place of 1 and 2, and the Poincar\'e inequality, we have
\begin{align}\label{nordirection1}
&\int_{Q^+_r}|D^{2m-1}u-(D^{2m-1}u)_{Q^+_r}|^2\,dx\,dt\nonumber \\
&\le C r^{2{\gamma}+d+2m} \|u\|^2_{\frac{a}{2m},a,Q^+_r}\le  C r^{2\gamma+d+2m}  \|u\|^2_{W^{1,2m}_p(Q_{1/2}^+)}\nonumber \\
&\le  C r^{2\gamma+d+2m}\|u\|^2 _{L_2(Q^+_1)}\le  C r^{2\gamma+d+2m} \|D^m_d u\|^{2}_{L_2(Q^+_1)}.
\end{align}
Let $v$ be the function in Lemma \ref{tech3}. Note that $v$ also satisfies \eqref{beq} and \eqref{bcd}. Then the inequality \eqref{nordirection1} holds with $v$ in place of $u$. Since $P(x)$ is a polynomial of degree $2m-1$, using Lemma \ref{tech3}, we have
\begin{align*}\nonumber
&\int_{Q^+_r}|D^{2m-1}u-(D^{2m-1}u)_{Q^+_r}|^2\,dx\,dt=
\int_{Q^+_r}|D^{2m-1}v-(D^{2m-1}v)_{Q^+_r}|^2\,dx\,dt\\
&\le C  r^{2\gamma+d+2m}  \int_{Q^+_1}|D_d^m v|^2\,dx\,dt
\le  C r^{2\gamma+d+2m} \int_{Q^+_1}|D^{m-1}D_d^m v|^2\,dx\,dt\\
&=C r^{2\gamma+d+2m} \int_{Q^+_1}|D^{m-1}D_d^mu-(D^{m-1}D_d^mu)_{Q^+_1}|^2
\,dx\,dt\\
&\le C  r^{2\gamma+d+2m} \int_{Q^+_1}|D^{2m-1}u-(D^{2m-1}u)_{Q^+_1}|^2\,dx\,dt.
\end{align*}
The lemma is proved.
\end{proof}

Let us turn to non-homogeneous systems. Consider that $u\in C_{loc}^\infty({\overline{\bR^{d+1}_+}})$ and satisfies
$$
\left\{\begin{aligned}
&\nonumber u_t+(-1)^mL_0u= f \quad \text{in} \quad Q^+_{4R},\\
& u=0,\,D_d u=0,\cdots, D^{m-1}_d u=0 \quad \text{on}\quad Q_{4R}\cap \{x_d=0\}.
\end{aligned}\right.
$$
We shall show the estimates of all derivatives up to order $2m$ with the exception of $D_d^{2m}u$.
Let $\zeta$ be the function in \eqref{zeta1}.
For $T=(4R)^{2m}$, consider the following system
$$
\left\{\begin{aligned}
&w_t+(-1)^mL_0w=\zeta (f-(f)_{B^+_{4R}}(t)) \quad \text{in}\quad  (-T,0)\times \mathbb{R}^d_{+},\\
&w=0,\,D_dw=0,\cdots,\,D_d^{m-1}w=0 \quad \text{on} \quad (-T,0)\times\{x_d=0\}\\
&{w=0\quad \text{on} \quad \{t=-T\}}.
\end{aligned}\right.
$$
We know from Theorem 2.6 in \cite{DK10} that the system above has a unique $W^{1,2m}_2$-solution which satisfies the following estimate
\begin{equation}
\|D^{2m}w\|^2_{L_2(Q^+_{4R})}\le C\|f-(f)_{B^+_{4R}}(t)\|^2_{L_2(Q^+_{4R})}\le CR^{2a+2m+d}[f]^{\ast2}_{a,Q^+_{4R}}.\label{w 2m}
\end{equation}
We again apply the mollification argument so that $w$ is smooth. Observe that  $v:=u-w$ satisfies
$$
\left\{\begin{aligned}
&v_t+(-1)^m L_0 v=(1-\zeta)f +\zeta(f)_{B^+_{4R}}(t)\quad\text{in}\quad Q^+_{2R},\\
&v=0,\,D_dv=0,\cdots,D^{m-1}_dv =0 \quad \text{on} \quad \{x_d=0\}\cap Q_{2R}.
\end{aligned}\right.
$$
Since $\zeta =1$ in $Q^+_{2R}$,  $\tilde v:=D_{x'}v$ satisfies
\begin{equation}        \label{v eq}
\left\{\begin{aligned}
&\tilde v_t+(-1)^m L_0 \tilde v=0\quad\text{in}\quad Q^+_{2R},\\
&\tilde v=0,\,D_d\tilde v=0,\cdots,D_d^{m-1}\tilde v =0 \quad \text{on} \quad \{x_d=0\}\cap Q_{2R}.
\end{aligned}\right.
\end{equation}
Thanks to  Lemma \ref{nordirection}, for any $X_0 \in \{x_d=0\}\cap Q_R$ we get
\begin{align}
& \int_{Q^+_r(X_0)}|D_{x^\prime}D^{2m-1} v-(D_{x^\prime}D^{2m-1}v)_{Q^+_r(X_0)}|^2 \,dx\,dt\nonumber\\
&\le  C(\frac{r}{R})^{2\gamma+d+2m}  \int_{Q^+_R(X_0)}|D_{x^\prime}D^{2m-1} v-(D_{x^\prime}D^{2m-1}v)_{Q^+_R(X_0)}|^2 \,dx\,dt\label{v estimate}.
\end{align}
Therefore, by the triangle inequality,
\begin{align}\nonumber
& \int_{Q^+_r(X_0)}|D_{x^\prime}D^{2m-1} u-(D_{x^\prime}D^{2m-1} u)_{Q^+_r(X_0)}|^2 \,dx\,dt\\\nonumber
& \le C(\frac{r}{R})^{2\gamma+d+2m} \int_{Q^+_R(X_0)}|D_{x^\prime}D^{2m-1} u-(D_{x^\prime}D^{2m-1} u)_{Q^+_R(X_0)}|^2\,dx\,dt\\
&\,\,+C\int_{Q^+_{2R}}|D^{2m}w|^2 \,dx\,dt\nonumber.
\end{align}
Combining \eqref{w 2m}, we obtain
\begin{align*}
& \int_{Q^+_r(X_0)}|D_{x^\prime}D^{2m-1} u-(D_{x^\prime}D^{2m-1} u)_{Q^+_r(X_0)}|^2 \,dx\,dt\\\nonumber
& \le C(\frac{r}{R})^{2\gamma+d+2m} \int_{Q^+_R(X_0)}|D_{x^\prime}D^{2m-1} u-(D_{x^\prime}D^{2m-1} u)_{Q^+_R(X_0)}|^2\,dx\,dt\\
&\,\,+CR^{2a+d+2m}[f]_{a,Q^+_{4R}}^{\ast 2}.
\end{align*}
Taking $\gamma>a$, from Lemmas \ref{tech}, \ref{holder}, and the corresponding interior estimates,  we obtain the estimate for the H\"{o}lder semi-norm,
\begin{equation*}
[D_{x^\prime}D^{2m-1}u]_{\frac{a}{2m},a,Q^+_R}\le C(\|D^{2m} u\|_{L_2(Q^+_{4R})}+[f]_{a,Q^+_{4R}}^{\ast}).
\end{equation*}
Let us turn to the case that the coefficients are functions of both $t$ and $x$, and the method of freezing the coefficients can still be implemented in this condition. We first consider the case that $L=\sum_{|\alpha|=|\beta|=m}A^{\alpha\beta}D^\alpha D^\beta$. Fix $y\in B_R^+$ and rewrite the system as follows
\begin{equation*}
u_t+(-1)^m L_{0y} u= f+(-1)^m(L_{0y}-L)u.
\end{equation*}

We use the same $\zeta$ as in \eqref{zeta1}, $T=(4R)^{2m}$, and $w$ is the unique $W_2^{1,2m}$-solution of the system
\begin{equation*}
w_t+(-1)^mL_{0y}w=\zeta (f-(f)_{B^+_{4R}}(t)) +\zeta(-1)^m(L_{0y}-L)u
\end{equation*}
in $(-T,0)\times \mathbb{R}_+^d$
with the Dirichlet boundary conditions $w=0,\,D_dw=0,\cdots,\,D_d^{m-1}w=0$ on  $(-T,0)\times\{x_d=0\}$ {and the zero initial condition at $t=-T$}.
 It is easy to obtain the estimate for $w$,
\begin{align*}
&\|D^{2m}w\|^2_{L_2(Q^+_{4R})} \le C\|f-(f)_{B^+_{4R}}(t)\|^2_{L_2(Q^+_{4R})}
+C\|(L_{0y}-L)u\|^2_{L_2(Q^+_{4R})}\nonumber\\
&\,{\le CR^{2a+2m+d}[f]^{\ast2}_{a,Q^+_{4R}}+ C R^{2a+2m+d}\sum_{|\alpha|=|\beta|=m}[A^{\alpha\beta}]_a^{\ast2}\|D^{2m} u\|^2_{L_\infty(Q^+_{4R})}}.
\end{align*}
Let $v=u-w$. Then $\tilde v:=D_{x'} v$ satisfies \eqref{v eq}, which implies \eqref{v estimate}.
Consequently, taking $\gamma>a$, we get the following inequality
\begin{align}                       \label{eq9.35}
[D_{x^\prime}D^{2m-1}u]_{\frac{a}{2m},a,Q^+_R} \le C([f]_{a,Q^+_{4R}}^{\ast}+\|D^{2m}u\|_{L_\infty(Q^+_{4R})}).
\end{align}
If the operator $L$ has lower-order terms, we can move all the lower-order terms to the right-hand side regarded as a part of $f$.
  Hence we proved the second part of Theorem \ref{nondiv theorem}, i.e.,
\begin{equation*}
[D_{x^\prime}D^{2m-1}u]_{\frac{a}{2m},a,Q^+_R}\le C([f]_{a,Q^+_{4R}}^{\ast}+\sum_{|\gamma|\le 2m}\|D^{\gamma}u\|_{L_\infty(Q^+_{4R})}).
\end{equation*}
The proof of Theorem \ref{nondiv theorem} is completed.

\begin{remark}                      \label{rem6.3}
We need more regularity assumptions to get the estimate of $[D_d^{2m}u]_{\frac{a}{2m},a, Q_R^+}$.
For example, if assuming $A^{\alpha\beta},f \in C^{\frac{a}{2m},a}$, then for $R \le1$,  by using a similar method one can show that
\begin{equation}
[D^{2m}_d u]_{\frac{a}{2m},a,Q_R^+}+[u_t]_{\frac{a}{2m},a, Q_R^+}\le C(\|f\|_{\frac{a}{2m},a, Q_{8R}^+}{+\|u\|_{L_2(Q_{8R}^+)}}).          \label{eq9.07}
\end{equation}
We give a sketched proof. First, let us assume that the coefficients are constants and $u$ is smooth and satisfies \eqref{beq} and \eqref{bcd}. In this case, we can differentiate the equation with respect to $t$ which means $u_t$ satisfies the same equation. From Lemma \ref{keybd}, we know that for any $r<R\le1$ and $X_0\in \{x_d=0\}\cap Q_R$, there exists a constant $C$ such that
\begin{equation}
\int_{Q^+_r(X_0)}|u_t|^2\,dx\,dt\le C(\frac{r}{R})^{d+4m}\int_{Q^+_R(X_0)}|u_t|^2\,dx\,dt.\label{t direction}
\end{equation}
We are now ready to show \eqref{eq9.07}.
For simplicity we only consider $L$ which consists of highest-order terms. We use the same $\zeta$ as in \eqref{zeta1}, $T=(4R)^{2m}$,  and fix $X_0 \in Q_{R}^+$. Let $w$ be the $W^{1,2m}_2$-solution of the system
\begin{equation*}
w_t+(-1)^mL_{X_0}w=\zeta (f-(f)_{Q^+_{4R}}(t)) +\zeta(-1)^m(L_{X_0}-L)u
\end{equation*}
in $(-T,0)\times \mathbb{R}^d_+$ with the Dirichlet boundary conditions $w=0,\,D_dw=0,\cdots,\,D_d^{m-1}w=0$ on $(-T,0)\times\{x_d=0\}$ {and the zero initial condition at $t=-T$}, where
\begin{equation*}
L_{X_0}=A^{\alpha\beta}(X_0)D^{\alpha}D^\beta.
\end{equation*}
Let $v=u-w$, which satisfies
$$
\left\{\begin{aligned}
&v_t+(-1)^mL_{X_0} v= (f)_{Q^+_{4R}}\quad \text{in}\quad Q^+_{2R},\\
&v=0,\,D_dv=0,\cdots, D^{m-1}_d v=0 \quad\text{on}\quad \{x_d=0\}\cap Q_{2R}.
\end{aligned}\right.
$$
Then \eqref{t direction} holds with $v$ in place of $u$. By the triangle inequality, for any $X_1 \in \{x_d=0\}\cap Q_R$ and $r\in (0,R)$,
\begin{align}
\int_{Q_r^+(X_1)}|u_t|^2\,dx\,dt \le C (\frac{r}{R})^{d+4m}\int_{Q_R^+(X_1)}|u_t|^2 \,dx\,dt+ C\int_{Q_{R}^+(X_1)}|w_t|^2\,dx\,dt.\label{u_t}
\end{align}
By Theorem 2.6 in \cite{DK10},
\begin{align}\nonumber
\|w_t\|^2_{L_2(Q^+_{4R})}& \le C \|f-(f)_{Q^+_{4R}}\|^2_{L_2(Q^+_{4R})}+\|(L_{X_0}-L)u\|^2_{L_2(Q^+_{4R})}\\
& \le CR^{2m+d+2a}\Big([f]^2_{\frac{a}{2m},a, Q^+_{4R}}+\sum_{|\alpha|=|\beta|=m}[A^{\alpha\beta}]^2_{\frac{a}{2m},a}
\|D^{2m}u\|^2_{L_\infty(Q^+_{4R})}\Big).\label{w_t}
\end{align}
We combine \eqref{u_t}, \eqref{w_t}, and the corresponding interior estimates to obtain a H\"{o}lder estimate,
 \begin{equation}
[u_t]_{\frac{a}{2m},a,Q_R^+} {\le C \big([f]_{\frac{a}{2m},a,Q_{4R}^+}+\sum_{|\alpha|=|\beta|=m}[A^{\alpha\beta}]_{\frac{a}{2m},a}\|D^{2m}u\|_{L_\infty(Q_{4R}^+)}
+\|u_t\|_{L_2(Q^+_{4R})}\big)}.      \label{u_t holder}
\end{equation}
It remains to estimate $D_d^{2m}u$. Let $\gamma= (0,0,\cdots, m)$,
\begin{equation*}
A^{\gamma\gamma}D^\gamma D^\gamma u=(-1)^{m+1}u_t -\sum_{(\alpha,\beta)\neq(\gamma,\gamma)}A^{\alpha\beta}D^\alpha D^\beta u +(-1)^mf.
\end{equation*}
Since $A^{\gamma\gamma}$ is positive definite, it follows easily that
\begin{align*}
[D^{2m}_d u]_{\frac{a}{2m},a,Q_R^+} \le C \big([u_t]_{\frac{a}{2m},a,Q^+_R}+\sum_{_{\quad|\alpha|= 2m}^{\alpha\neq (0,\cdots,2m)}}\|D^\alpha u\|_{\frac{a}{2m},a,Q^+_R}+[f]_{\frac{a}{2m},a,Q_R^+}\big).
\end{align*}
Plugging \eqref{eq9.35} and \eqref{u_t holder} into the inequality above and using the interpolation inequalities, we immediately prove \eqref{eq9.07}.
\end{remark}


\bibliographystyle{plain}

\begin{thebibliography}{1}
\bibitem{ADN64} \textsc{Agmon S., Douglis A., Nirenberg L.}: Estimates near the boundary for solutions of elliptic partial differential equations satisfying general boundary
conditions, I, \textit{Comm. Pure Appl. Math.} \textbf{12}, 623--727  (1959); II, ibid., \textbf{17}, 35--92  (1964).

\bibitem{BA69} \textsc{Brandt A.}: Interior Schauder estimates for parabolic differential-(or difference-) equations via the maximum principle, \textit{Israel J. Math.} \textbf{7} (1969), 254--262.

\bibitem{Sb12} \textsc{Boccia S.}: Schauder estimates for solutions of high-order parabolic systems, \textit{preprint}.


\bibitem{Cam66} \textsc{Campanato S.}: Equazioni paraboliche del secondo ordine e spazi $\l^{2,\theta}(\Omega,\delta)$, \textit{Ann. Mat. Pura. Appl.}  \textbf{73} (1966), 55--102.

\bibitem{DK10} \textsc{Dong H., Kim D.}: On the $L_p$-solvability of higher order parabolic and elliptic systems with BMO coefficients,  \textit{Arch. Rational Mech. Anal.} \textbf{199} (2011), no. 3, 889--941.

\bibitem{DK11} \textsc{Dong H., Kim D.}: Higher order elliptic system in Sobolev spaces with partial BMO coefficients, \textit{J. Funct. Anal.} \textbf{261} (2011), no. 11, 3279--3327.



\bibitem{FA} \textsc{Friedman A.}: \textit{Partial differential equations of parabolic type}, Prentice-Hall, Englewood cliffs, N.J., 2008.

\bibitem{Gia93} \textsc{Giaquinta M.}:  \textit{Introduction to regularity theory for nonlinear elliptic systems}, Lectures in Mathematics ETH Z\"urich. Birkh\"auser Verlag, Basel, 1993.

\bibitem{L92} \textsc{Lieberman G.}: Intermediate Schauder Theory For second order parabolic equations. IV. Time irregularity and regularity, \textit{Differential and Integral Equations} \textbf{5} (1992), no. 6, 1219--1236.

\bibitem{L96} \textsc{Lieberman G.}: \textit{Second order parabolic differential equations}, Word Scientific Publishing Co. Pte. Ltd, Singapore-New Jersey-London-Hong Kong, 1996.


\bibitem{L00} \textsc{Lorenzi L.}: Optimal Schauder estimates for parabolic problems with data measurable with respect to time, \textit{SIAM J. Math. Anal.} \textbf{32} (2000), no. 3, 588--615.


{\bibitem{LA95} \textsc{Lunardi, A.}: \textit{Analytic semigroups and optimal regularity in parabolic problems}, Progress in Nonlinear Differential Equations and their Applications, 16. Birkh\"auser Verlag, Basel, 1995.}

\bibitem{BK80} \textsc{Knerr B.}: Parabolic interior Schauder estimates by the maximum principle, \textit{Arch. Rational Math. Anal.} \textbf{75} (1980), 51--58.

\bibitem{Kry96} \textsc{Krylov N. V.}: \textit{Lectures on elliptic and parabolic equations in H\"{o}lder spaces}, American Mathematical Society, Providence, RI, 1996.


\bibitem{KP10} \textsc{Krylov N. V., Priola E.}: Elliptic and parabolic second-order PDEs with growing coefficients, \textit{Comm. Partial Differential Equations} \textbf{35} (2010), no. 1, 1-22.



\bibitem{Sh96} \textsc{Schlag W.}: Schauder and $L^p$ estimates for parabolic system via Campanato spaces, \textit{Comm. Partial Differential Equations} \textbf{21} (1996), no. 7-8, 1141--1175.

\bibitem{SL97} \textsc{Simon L.}: Schauder estimates by scaling, \textit{Calc. Var. Partial Differential Equations} \textbf{5} (1997), no. 5,  391--407.

\bibitem{SV88} \textsc{Sinestrari E., von Wahl W.}: On the solutions of the first boundary value problem for the linear parabolic equations, \textit{Proc. Roy. Soc. Edinburgh Sect. A} \textbf{108} (1988), no. 3-4, 339--355.

\end{thebibliography}
\def\cprime{$'$}\def\cprime{$'$} \def\cprime{$'$} \def\cprime{$'$}
  \def\cprime{$'$} \def\cprime{$'$}

\end{document}